\documentclass[hidelinks,final]{siamart190516}

\usepackage{lipsum}
\usepackage{amsfonts}
\usepackage{graphicx}
\usepackage[caption=false]{subfig}
\usepackage{epstopdf}
\usepackage{algorithmic}
\usepackage{booktabs} 
\usepackage{threeparttable}
\usepackage{multirow}
\ifpdf
  \DeclareGraphicsExtensions{.eps,.pdf,.png,.jpg}
\else
  \DeclareGraphicsExtensions{.eps}
\fi

\newsiamremark{remark}{Remark}
\newsiamremark{hypothesis}{Hypothesis}
\crefname{hypothesis}{Hypothesis}{Hypotheses}
\newsiamthm{claim}{Claim}
\newsiamthm{assumption}{Assumption}
\newsiamthm{example}{Example}
\newsiamthm{conjecture}{Conjecture}

\headers{global optimum of a class of quartic minimization problem}{P. Huang, Q. Yang, Y. Yang}

\title{Finding the global optimum of a class of quartic minimization problem\thanks{Submitted to the editors DATE.
\funding{The first author was supported by the Tianjin Graduate Research and Innovation Project 2019YJSB040. The second author was supported by the National Natural Science Foundation of China Grant 11671217 and 12071234. The third author was supported by the National Natural Science Foundation of China Grant 11801100 and the Fok Ying Tong Education Foundation Grant 171094.}}}

\author{Pengfei Huang\thanks{School of Mathematical Sciences, Nankai University, Tianjin, P.R. China. (\email{huangpf@mail.nankai.edu.cn}).}
\and
Qingzhi Yang\thanks{Corresponding author, School of Mathematical Sciences and LPMC, Nankai University, Tianjin, P.R. China. (\email{qz-yang@nankai.edu.cn}).}
\and
Yuning Yang\thanks{College of Mathematics and Information Science, Guangxi University, Nanning, Guangxi, P.R. China. (\email{yyang@gxu.edu.cn}).}
}

\usepackage{amsopn}

\ifpdf
\hypersetup{
  pdftitle={Finding the global optimum of a class of quartic minimization problem},
  pdfauthor={Pengfei Huang, Qingzhi Yang, Yuning Yang}
}
\fi
\begin{document}
\maketitle

\begin{abstract}
    We consider a special nonconvex quartic minimization problem over a single spherical constraint, which includes  the discretized energy functional minimization problem of non-rotating Bose-Einstein condensates (BECs) as one of the important applications. Such a problem is studied by exploiting its characterization as a nonlinear eigenvalue problem with eigenvector nonlinearity (NEPv), which admits a unique nonnegative eigenvector, and this eigenvector is exactly the global minimizer to the quartic minimization.
    With these properties, any algorithm converging to the nonnegative stationary point of this optimization problem   finds its global minimum, such as the regularized Newton (RN) method. In particular, we   obtain the global convergence to global optimum of the inexact alternating direction method of multipliers (ADMM) for this problem. Numerical experiments for applications in non-rotating BEC validate our theories.
\end{abstract}



\section{Introduction}
In this paper, we consider the following nonconvex quartic optimization problem over a   spherical constraint:
\begin{equation}\label{equ:different}
\left\{
\begin{array}{lrc}
\underset{x\in \mathbb R^n}{\min} \quad \frac{\alpha}{2}\sum\limits_{i=1}^{n} x_i^4+x^{T}Bx\\
{\rm s.t.} \quad \|x\|_2^2=1,
\end{array}
\right.
\end{equation}
where $\alpha>0$ is a fixed constant and $B$ is an irreducible $n$ by $n$ symmetric $M$-matrix, with positive diagonal entries and nonpositive off-diagonal entries. An important application of this model is to find the ground state of the non-rotating Bose-Einstein condensation (BEC), which is usually defined as the minimizer of the energy functional minimization problem. After a suitable discretization, the matrix $B$ expresses as the sum of the discretized Laplacian operator and a positive diagonal matrix. See \cite{bao2013mathematical} and references therein for details.

 BEC has attracted great interest in the atomic, molecular and optical physics community and condense matter community \cite{pethick2008bose,fetter2009rotating}. As one of the major problems in the study of BEC, there are already several popular numerical methods that work well to compute the ground state. One class of these methods has been designed for  finding the smallest eigenvalue and corresponding eigenvector of the nonlinear eigenvalue problem with eigenvector nonlinearity (NEPv), which arises from the Gross-Pitaevskii equation (GPE), such as self-consistent field iteration (SCF) \cite{cai2018eigenvector}, full multigrid method \cite{jia2016full}, etc.. The second class deals with the nonconvex constrained minimization problem \eqref{equ:different} and its continuous version; see \cite{wu2017regularized,bao2004computing} and references therein. In fact, it is easy to show that the NEPv is the first-order optimality condition for this minimization problem. However, to the best of our knowledge, little of them give a theoretical guarantee about whether these methods find the best solution for both NEPv and optimization problems.

Although for algorithms solving nonconvex optimization problems, it is generally difficult to guarantee convergence to a global optimum, many approaches nevertheless can be applied to solve certain nonconvex problems with effective numerical results. To deal with the orthogonality including the spherical constraint, the constraint preserving algorithm has been proposed based on the manifold optimization theory, where a curvilinear search approach was introduced combined with Barzilai-Borwein step size \cite{wen2013feasible, hu2018adaptive}. In those work, convergence to a stationary point was established under some assumption. However, the manifold techniques are sophisticated and complex for implementation if additional constraints are imposed in general. Hu et al. \cite{hu2016note} showed the NP-completeness of \cref{equ:different} with a Hermitian $B$ by establishing its connection to the partition problem. And they solved it approximately by using semidefinite programming (SDP) relaxations, which is known to be time-consuming if the problem is large. The splitting method using Bregman iteration, which covers the alternating direction method of multipliers (ADMM), was also applied to solve orthogonality constrained problems \cite{lai2014splitting} without convergence analysis while presenting numerical results quite well. Zhang et al. \cite{zhang2019geometric} offered the geometric analysis of \cref{equ:different} when $B$ is with different structures, such as diagonality and rank-one. They also obtained meaningful results for general matrix $B$ utilizing fourth-order optimality conditions and strict-saddle property. We believe that for the special case considered in this paper, more specified results can be reached with simpler proof and convergence to the global optimum can be achieved for certain algorithms.

Triggered by the nice numerical results with BEC for both NEPv and \cref{equ:different}, we study the property of the first-order necessary condition for \cref{equ:different}, to give a hint to the global optimum. It is well known that, without rotation, the ground state can be taken as a real nonnegative function, and it corresponds to the smallest eigenvalue of NEPv in physics or partial differential equations theory \cite{cances2010numerical}. Taking advantage of the special structure, we give a rather simple proof from the linear algebraic point of view for such kinds of results. A similar idea was used in research about optimization of the trace ratio by Bai, et al. \cite{bai2018robust}. We first prove that for structured $B$, the NEPv has a unique positive eigenvector corresponding to the smallest eigenvalue, which is exactly a global optimum for \cref{equ:different}. Then we obtain convergence to the global optimizer of the regularized Newton (RN) method \cite{wu2017regularized,wen2013adaptive} trivially, and provide an analysis on the global convergence to a global minimum for ADMM based on the work of Wang, et al. \cite{wang2019global}.

In this paper, we begin the presentation in \cref{sec:pre} with preliminary. In \cref{sec:nepv}, we exploit the properties of NEPv corresponding to \cref{equ:different} and establish the relationship of the positive stationary point and the global minimum. In \cref{sec:algorithms}, firstly, we show that it is easy to obtain the convergence to a global minimizer of the RN method. Secondly, we state the standard ADMM and then derive global convergence to a global optimum for it in the inexact version. The application examples on non-rotating BEC problem and numerical performance are given in \cref{sec:numerical}. Concluding remarks are in \cref{sec:conclude}.

\section{Preliminary}
\label{sec:pre}

In this section, we define the notations and sort out some basic definitions and facts, which will be used in the subsequent analysis.
Throughout the paper, we follow the notations commonly used in numerical linear algebra. A vector $x>(\ge)0 ~(x\in \mathbb{R}^n)$, stands for $x_i>(\ge)0~(\forall i\in[n])$. $[n]$ denotes $\{1,2,\cdots,n\}$. We use bold lowercase letter $\textbf{x}$ for the spatial coordinate vector. $\|\cdot\|$ is the norm $\|\cdot\|_2$ for vectors and matrices. For $x\in \mathbb{R}^n$, $|x|=(|x_1|,|x_2|,\cdots,|x_n|)^T$, $\sqrt{x}=(\sqrt{x_1},\sqrt{x_2},\cdots,\sqrt{x_n})^T$. $B\succeq 0$ means that $B$ is positive semidefinite. $\lambda_{min}(B)$ and $\lambda_{max}(B)$ are the smallest and largest eigenvalue of $B$, respectively. "$\otimes$" denotes the kronecker product.

Let $\mathcal{A}$ be a
fourth-order diagonal tensor with all diagonal entries are one, then $\mathcal{A}x^4=\sum_{i=1}^n x_i^4$, $\mathcal{A}x^3=(x_1^3,~x_2^3,~\cdots,~x_n^3)^T$ is a vector, and $\mathcal{A}x^2$ is a diagonal matrix with $(x_1^2,~x_2^2,~\cdots,~x_n^2)$ as diagonal entries. For $x\in \mathbb{C}^n$, $\mathcal{A}x^Hxx^Hx=\sum_{i=1}^n |x_i|^4$.
\begin{definition}[Irreducibility/Reducibility \cite{varga1962iterative}]
A matrix $B\in \mathbb{R}^{n\times n}$ is called reducible, if there exists a nonempty proper index subset $I\subset[n]$, such that
$$b_{ij}=0,\quad \forall i\in I,\quad\forall j\notin I.$$
If $B$ is not reducible, then we call $B$ irreducible.
\end{definition}
\begin{definition}[$M$-matrix \cite{varga1962iterative}]
A matrix $B\in\mathbb{R}^{n\times n}$ is called an $M$-matrix, if $B=sI-A$, where $A$ is nonnegative, and $s\ge\rho(A)$. Here $\rho(A)$ is the spectral radius of $A$.
\end{definition}

\section{Nonlinear eigenvalue problems and global optimum}
\label{sec:nepv}

In this section, we characterize the spherical constraint minimization problem \cref{equ:different} by a nonlinear eigenvalue problem with eigenvector nonlinearity. It is in fact the first-order necessary condition of this constrained problem, since the linear independence constraint qualification (LICQ) holds at the local solution of \cref{equ:different} \cite{nocedal2006numerical}. First, let us define its Lagrangian function with multiplier $\lambda$:
\begin{equation*}
L(x,\lambda)=\frac{\alpha}{2}\mathcal{A}x^4+x^{T}Bx-\lambda(x^Tx-1).
\end{equation*}
Then we can get the following nonlinear eigenvector problem (NEPv):
\begin{equation}\label{equ:kkt}
\left\{
\begin{array}{lrc}
\alpha\mathcal{A}x^3+Bx=\lambda x\\
\|x\|^2=1.
\end{array}
\right.
\end{equation}
Any $(\lambda,~x)$ satisfying \cref{equ:kkt} is called an eigenpair of the NEPv. $\lambda$ and $x$ are the corresponding eigenvalue and eigenvector, respectively. It is obvious that $x$ is an eigenvector is equivalent to $x$ is a stationary point of \cref{equ:different}. There always exists an eigenpair for \cref{equ:kkt}, since \cref{equ:different} is to minimize a continuous function over a compact set.

In the following of this section, we give a rather simple proof to specify to which eigenvalue $x$ corresponds, if $x$ is nonnegative. Before that, we need some assumptions for the structure of \cref{equ:different}. Except stated otherwise, the throughout article will be discussed under the following assumption.
\begin{assumption}\label{ass:structure}
$B$ is an $n\times n$ symmetric irreducible $M$-matrix.
\end{assumption}

According to the definition of $M$-matrix, $B$ is positive semidefinite and the off-diagonal entries of $B$ are nonpositive.
\begin{remark}
The assumption for $B$ to be positive semidefinite is reasonable, since adding a quadratic term $\gamma \|x\|_2^2~(\forall \gamma \ge0)$ does not change \cref{equ:different}. $M$-matrix arises naturally in the numerical solution of elliptic partial differential equations. The Laplace operator with finite difference discretization is an $M$-matrix.
\end{remark}

Here we present two examples that satisfy \cref{ass:structure}.
\begin{example}
$B$ is a tridiagonal matrix with positive diagonal entries. If the sub-diagonal elements of $B$ are negative, then $B$ is irreducible. Furthermore, in this case, the nonnegative eigenvector of \eqref{equ:kkt} has no zero entry.
\begin{proof}
We first show that $B$ is irreducible. Suppose on the contrary that $B$ is reducible, according to the definition, there exists a nonempty proper index subset $I\subset[n]$, such that $b_{ij}=0,~\forall i\in I,~\forall j\notin I$. Let $k$ be the largest index in $I$, without loss of generality, we assume $k<n$. Then, $b_{kk+1}=0$. It contradicts the assumption that the sub-diagonal entries of $B$ are negative.

Suppose the nonnegative eigenvector $x\in \mathbb{R}^n$ has some $x_i=0,~i\in [n]$. Let $x_{i-1}=0$ if $i=1$ and $x_{i+1}=0$ if $i=n$, then according to \eqref{equ:kkt},
$$\alpha x_i^3+b_{ii}x_i+b_{ii-1}x_{i-1}+b_{ii+1}x_{i+1}=0.$$
Thus, $x_{i-1}=0$, $x_{i+1}=0$, and recursively, $x=0$, which contradicts the definition of eigenvector.
\end{proof}
\end{example}

\begin{example}
Consider the Bose-Einstein condensation (BEC) problem discretized by finite difference method, where the two-dimensional space domain $D=[0,1]\times[0,1]\backslash[0.5,1]\times[0.5,1]$ is "L"-like and the boundary value is zero. We divide $D$ evenly along both two directions with step $h=\frac{1}{n}$, and choose $n=6$. Then the discretized Laplacian operator will be
\begin{equation*}
\begin{split}
\tilde{B}=n^2\cdot
\begin{bmatrix}
B_1&\Sigma& & & &\\
\Sigma&B_1&\Sigma& & &\\
&\Sigma&B_1&\Sigma& &\\
& &\Sigma&B_2&-I_{5}\\
& & &-I_{5}&B_2
\end{bmatrix}
,~ where ~\Sigma=
\begin{bmatrix}
-1& & & & &\\
&-1& & & &\\
& &0& & &\\
& & &0& &\\
& & & &0 &
\end{bmatrix},\\
B_1=
\begin{bmatrix}
4&-1& & & &\\
-1&4& 0& & &\\
&0&0 &0 & &\\
& &0&0 &0 &\\
& & & 0&0 &
\end{bmatrix}
,\quad B_2=
\begin{bmatrix}
4&-1& & & &\\
-1&4&-1& & &\\
&-1&4&-1& &\\
& & -1&4&-1\\
& & & -1& 4
\end{bmatrix}.
\end{split}
\end{equation*}
That is, the columns and rows corresponding to $[0.5,1]\times[0.5,1]$ are zero. Removing these zero columns and rows will not affect the value of \cref{equ:different}, we obtain
\begin{equation*}
B=n^2\cdot
\begin{bmatrix}
\tilde{B}_1&-I_2& & &\\
-I_2&\tilde{B}_1&-I_2& &\\
&-I_2&\tilde{B}_1&-I_{2\times 5}&\\
& &-I_{5\times2}&B_2&-I_5\\
& & &-I_5&B_2
\end{bmatrix}+V,
\end{equation*}
where $\tilde{B}_1=(\begin{smallmatrix}4&-1\\-1&4\end{smallmatrix})$, and $V$ is a positive diagonal matrix dicretized from the external trapping potential (given in \cref{sec:numerical}). Here, $I_{m\times n}$ stands for an $m\times n$ matrix with diagonal entries being one and other elements being zero.
For the discretized non-rotating BEC problem, \cref{ass:structure} is satisfied through the recursive adjacency in the Laplacian operator.
\end{example}

Although we hope that \cref{ass:structure} to be satisfied by rotating BEC model, it seems it is not the case.
\begin{example}
Consider the rotating BEC problem and $\textbf{x}=(x,~y)^T\in \mathbb{R}^2$. There is a term looks like $-\Omega\bar{\phi}(\textbf{x})L_z\phi(\textbf{x})$, where $\Omega$ is an angular velocity, $L_z=-i(x\partial_y-y\partial_x)$. Let the two-dimensional space domain $D=[0,1]\times [0,1]$. We divide $\Omega$ evenly along both two directions with step $h=\frac{1}{n}$, and choose $n=4$. Through the finite difference method, $\partial_y(\phi(x_i,y_j))=\frac{\phi(x_i,y_{j+1})-\phi(x_i,y_{j-1})}{2h}$. We obtain discretized
$$L_z=-i(I_x\otimes D_y-D_x\otimes I_y),$$
where $I_x=I_y=h\left(\begin{smallmatrix}1&&\\&2&\\&&3\end{smallmatrix}\right)$ are diagonal matrices,  $D_x=D_y=\frac{1}{2h}\cdot\left(\begin{smallmatrix}0&1&0\\-1&0&1\\0&-1&0\end{smallmatrix}\right)$ are skew-symmetric matrices.  Now \cref{equ:different} should be considered over complex field and $B$ corresponds to the sum of the discretized Laplace operator, a diagonal matrix and $-\Omega L_z$. Thus, \cref{ass:structure} is not satisfied.
\end{example}

After some intuitive illustration of the structure of $B$, we move on to discuss the properties of the NEPv. We first define the geometric simplicity for the NEPv as Chang et al. \cite{chang2008perron}.
\begin{definition}[geometric simplicity of the NEPv]
Let $\lambda$ be an eigenvalue of NEPv \cref{equ:kkt}. We say that $\lambda$ is geometrically simple if the maximum number of linearly independent eigenvectors corresponding to $\lambda$ equals one. If we restrict the eigenvector $x$ on the real space, then we call $\lambda$ real geometrically simple; if $x$ is restricted on the complex space, then $\lambda$ is called complex geometrically simple.
\end{definition}
\begin{lemma}\label{lem:eigen}
Under \cref{ass:structure}, the eigenpair $(\lambda, ~x)$ has the following properties:
\begin{enumerate}
\item There exists an eigenpair $(\lambda,x)$ with $x\ge 0$.
\item The eigenvalue with an eigenvector $x\ge0$ is unique, and $x$ contains no zero entries, that is, $x>0$. Denote this $\lambda$ as $\lambda_0$.
\item $\lambda\ge \lambda_0>0$, for all eigenvalue $\lambda$, and $\lambda_0$ is real geometrically simple. That is, the smallest eigenvalue has and only has two eigenvectors, a nonnegative one and a nonpositive one.
\end{enumerate}
\end{lemma}
\begin{proof}
We prove this lemma in the case $\alpha=1$, while it is obvious that the proof can be generalized for all $\alpha>0$.

\textbf{1.} For any $x\in \mathbb{R}^n$,
\begin{equation}\label{equ:proof_nonnegative}
\begin{aligned}
\frac{1}{2}\mathcal{A}x^4+x^{T}Bx&=\frac{1}{2}\sum_{i}x_i^4+\sum_{i,j}b_{ij}x_ix_j\\
(\text{nonpositive off-diagonal of $B$})&\ge \frac{1}{2}\sum_{i}|x_i|^4+\sum_{i,j}b_{ij}|x_i||x_j|\\
&=\frac{1}{2}\mathcal{A}|x|^4+|x|^{T}B|x|.
\end{aligned}
\end{equation}
Thus, \cref{equ:different} has a nonnegative solution, which is a nonnegative eigenvector of \eqref{equ:kkt}.

\textbf{2.} For the positiveness of $x$, suppose $x$ is a nonnegative eigenvector, and there exists a nonempty set $I\subset[n]$, $\overline{I}=[n]\backslash I$ such that $x_i>0~(i\in \overline{I})$ and $x_i=0~(i\in I)$. For any $k\in I$, $b_{kj}=0~(\forall~j\notin I)$ follows from $\sum_{j\neq k}b_{kj}x_j=0$, since $b_{kj}\le 0$ when $i\neq j$. It contradicts the assumption that $B$ is irreducible.

For the uniqueness of $\lambda$ with positive eigenvector, suppose $(\lambda,x)$,~$(\mu,y)\in \mathbb{R}\times \mathbb{R}_+^{n}$ are two eigenpairs, then $x>0$, $y>0$, $\|x\|=\|y\|=1$. Denote $t=\min\limits_{i}\{\frac{x_i}{y_i}\}$, then $0<t\le 1$ and $x\ge ty$ with $x_k=ty_k$ for some k. We have
\begin{equation}\label{equ:proof2}
\begin{aligned}
\lambda x_k=x_k^3+b_{kk}x_k+\sum_{j\neq k}b_{kj}x_j &\le (ty_k)^3+b_{kk}(ty_k)+\sum_{j\neq k}b_{kj}(ty_j)\\
(0<t\le 1,~y_k>0)&\le ty^3_k+t(By)_k\\
&=t\mu y_k,
\end{aligned}
\end{equation}
then $\lambda \le \mu$, and as the same we can get $\mu\le\lambda$. So $\lambda=\mu$.

\textbf{3.} Suppose $(\mu,y)$ is an eigenpair; then $\mu=\mathcal{A}y^4+y^TBy>0$ since $B\succeq 0$. We first prove that $\mu|y|=|\mathcal{A}y^3+By|\ge\mathcal{A}|y|^3+B|y|$. Since $\mu y_i=y_i^3+b_{ii}y_i+\sum_{j\neq i}b_{ij}y_j$, if $y_i\ge0$,
\begin{equation*}
\begin{aligned}
\mu |y_i|&=y_i^3+b_{ii}y_i+\sum_{j\neq i}b_{ij}y_j\\
&\ge|y_i|^3+(B|y|)_i;
\end{aligned}
\end{equation*}
if $y_i\le0$, we can prove it similarly. The remaining part for $\mu\ge\lambda_0$ is analogous to the above proof of \cref{equ:proof2}, so we just omit it here.

Suppose $(\lambda_0,~y)$ is an eigenpair; we also have $\lambda_0|y|\ge\mathcal{A}|y|^3+B|y|$. According to \textbf{2.}, there is an eigenvector $x$ corresponding to $\lambda_0$ which is positive. There is a $t$ such that $0<t\le 1$, $x\ge t|y|\ge 0$ and $x_k=t|y_k|$ for some $k$. Then
\begin{equation}\label{equ:proof_unique}
\begin{aligned}
\lambda_0 x_k=x_k^3+b_{kk}x_k+\sum_{j\neq k}b_{kj}x_j &\le (t|y_k|)^3+b_{kk}(t|y_k|)+\sum_{j\neq k}b_{kj}(t|y_j|)\\
&\le t|y_k|^3+t(B|y|)_k\\
&\le t\lambda_0 |y_k|.
\end{aligned}
\end{equation}
Since $\lambda_0 x_k=t\lambda_0 |y_k|$, the first inequality in \cref{equ:proof_unique} equals. This leads to
$$\sum\limits_{j\neq k}b_{kj}(x_j-t|y_j|)=0.$$
We obtain that $x-t|y|=0$ as the proof for the positiveness of the nonpositive eigenvector.
Thus, it holds that
\[x=|y|, ~\mathcal{A}|y|^3+B|y|=\lambda_0|y|, ~\|y\|=1, ~|y|>0.\]

If $y_i>0$ for some $i$, then
\begin{equation*}
(\mathcal{A}y^3+By)_i=y_i^3+b_{ii}y_i+\sum_{j\neq i}b_{ij}y_j=|y_i|^3+b_{ii}|y_i|+\sum_{j\neq i}b_{ij}|y_j|,
\end{equation*}
which leads to the conclusion $y=x>0$. For $y_i<0$, we can get $y=-x<0$ similarly. Thus, $\lambda_0$ is real geometrically simple.
\end{proof}
\begin{remark}
From the proof of \cref{lem:eigen}, we can see that it is also possible to derive the NEPv characterization in the complex case,
\begin{equation}\label{equ:complex}
\left\{
\begin{array}{lrc}
\underset{x\in \mathbb{C}^n}{\min} \quad \frac{\alpha}{2}\mathcal{A}x^Hxx^Hx+x^{H}Bx\\
{\rm s.t.}\quad \|x\|^2=1.
\end{array}
\right.
\end{equation}
The lemma and theorem discussed above can be established using similar arguments. In particular, the smallest eigenvalue  has complex geometrically simplicity.
\end{remark}
\begin{remark}\label{rem:sdp}
We can prove the existence of the nonnegative optimum from a different aspect. In regard of the semidefinite relaxation of \cref{equ:different} as the following form,
\begin{equation}\label{equ:sdp1}
\left\{
\begin{array}{lrc}
\min\quad\frac{\alpha}{2}\sum\limits_{i=1}^nX_{ii}^2+\left<B, X\right>\\
{\rm s.t.} \quad tr(X) = 1\\
\qquad X\succeq 0,
\end{array}
\right.
\end{equation}
where $\left<B,X\right>=tr(B^TX)$ and $tr(X)$ denotes the trace of $X$. Yang et al. \cite{yang2019numerical} observed that the relaxation is tight based on \cref{lem:tight}, which is stronger than the tightness obtained by Hu et al. \cite{hu2016note} for general symmetric $B$.
\begin{lemma}\label{lem:tight}
When the off-diagonal entries of $B$ are nonpositive, \cref{equ:different} and \cref{equ:sdp1} are equal. If $X$ is an optimum for \cref{equ:sdp1}, then $x=\sqrt{diag(X)}$ is an optimum of \cref{equ:different}.
\begin{proof}
The proof follows \cite[Theorem 2.1]{yang2012solving}. If $x$ is a feasible point of \cref{equ:different}, then $xx^T$ is feasible for \cref{equ:sdp1}. Let v\cref{equ:different} and v\eqref{equ:sdp1} denote the optimal values of \cref{equ:different} and \cref{equ:sdp1}, respectively, we have $v\cref{equ:sdp1}\le v\cref{equ:different}$.
On the other hand, if $X$ is a feasible point of \cref{equ:sdp1}, let $x_i=\sqrt{X_{ii}}$; then $x$ is also a feasible point of \cref{equ:different} and
\begin{eqnarray*}
\frac{\alpha}{2}\mathcal{A}x^4+x^{T}Bx& = &\frac{\alpha}{2}\sum\limits_{i=1}^n\sqrt{X_{ii}}\sqrt{X_{ii}}\sqrt{X_{ii}}\sqrt{X_{ii}}+\sum\limits_{i,j=1}^nb_{ij}\sqrt{X_{ii}}\sqrt{X_{jj}}\\
&\le& \frac{\alpha}{2}\sum\limits_{i=1}^n X_{ii}^2+\sum\limits_{i,j=1}^nb_{ij}X_{ij}\\
&=&\frac{\alpha}{2}\sum\limits_{i=1}^n X_{ii}^2+\left<B, X\right>.
\end{eqnarray*}
Thus, $v\cref{equ:different}\le v\cref{equ:sdp1}$. The claimed results then follow.
\end{proof}
\end{lemma}

Furthermore, we can obtain the uniqueness of the nonnegative optimizer of \cref{equ:different}.
\begin{theorem}
The nonnegative optimizer of \cref{equ:different} is unique.
\end{theorem}
\begin{proof}
Let $x^*$ be a nonnegative optimizer of \cref{equ:different}. According to \cref{lem:tight}, $X=x^*(x^*)^T$ is a optimizer of \cref{equ:sdp1}. Since  \cref{equ:sdp1} is a convex optimization problem and the objective function is strongly convex, the optimizer is unique. Thus $x^*$ is unique.
\end{proof}
\end{remark}
\begin{remark}
Motivated by \cite{cances2010numerical}, we can also obtain the uniqueness of the nonnegative eigenvector based on \cref{lem:equal} and \cref{lem:convex}.
\begin{lemma}\label{lem:equal}
\cref{equ:different} has the same optimal value as the following problem. Moreover, $x$ is a nonnegative eigenvector of \cref{equ:kkt} if and only if that $y$, given by $y_i=x_i^2~(\forall i\in[n])$, is a positive stationary point of \cref{equ:reform} as follows.
\begin{equation}\label{equ:reform}
\left\{
\begin{array}{lrc}
\underset{y\in \mathbb{R}^n}{\min} \quad f(y)=\frac{\alpha}{2}\sum\limits_{i=1}^{n} y_i^2+\sqrt{y}^{T}B\sqrt{y}\\
{\rm s.t.}\quad \sum\limits_{i=1}^ny_i=1,~y\ge 0.
\end{array}
\right.
\end{equation}
\end{lemma}
\begin{proof}
If $x\ge 0$ is a feasible point of \cref{equ:different}, $y$ with $y_i=x_i^2,~\forall i\in[n]$, is a feasible point of \cref{equ:reform}, and the converse is also true. Thus, \cref{equ:different} and \cref{equ:reform} have the same optimal value.

According to \cref{lem:eigen}, the nonnegative eigenvector has no zero entries. The stationary point of \cref{equ:reform} is defined by the following conditions. It is well defined when $y>0$.
\begin{equation*}
\begin{aligned}
\alpha y_i+\sum\limits_{j=1}^nb_{ij}\frac{\sqrt{y_j}}{\sqrt{y_i}}-\lambda-\mu_i=0,&~\forall i\in[n];\\
\sum\limits_{i=1}^ny_i=1;&\\
\mu_i\ge 0,~\mu_iy_i=0,&~\forall i\in[n].
\end{aligned}
\end{equation*}
Compare it with \cref{equ:kkt}, the desired result follows.
\end{proof}
\begin{lemma}\label{lem:convex}
The objective function $f(y)$ of \cref{equ:reform} is strictly convex over the convex set $\mathcal{S}=\{y\ge0|\sum_{i=1}y_i=1,~y\in \mathbb{R}^n\}$. Thus \cref{equ:reform} has unique positive stationary point, which is exactly the optimizer.
\end{lemma}
\begin{proof}
Assume $y>0$, we have
\[(\nabla f(y))_i=\alpha y_i+\sum\limits_{j=1}^nb_{ij}\frac{\sqrt{y_j}}{\sqrt{y_i}}.\]
Then we have
\begin{equation*}
(\nabla^2 f(y))_{ij}=
\left\{
\begin{aligned}
\alpha-\sum\limits_{j\neq i}\frac{1}{2}\frac{b_{ij}\sqrt{y_j}}{y_i\sqrt{y_i}},&~j=i\\
\frac{1}{2}\frac{b_{ij}}{\sqrt{y_iy_j}},&~j=i.
\end{aligned}
\right.
\end{equation*}
Thus, for any $z\in \mathbb{R}^n$,
\begin{equation*}
\begin{aligned}
z^T\nabla^2f(y)z&=\sum\limits_{i=1}^n\alpha z_i^2+\sum\limits_{i=1}^n\left(\sum\limits_{j\neq i} \frac{1}{2}\frac{b_{ij}}{\sqrt{y_iy_j}}z_iz_j-\frac{1}{2}\frac{b_{ij}\sqrt{y_j}}{y_i\sqrt{y_i}}z_i^2\right)\\
&=\sum\limits_{i=1}^n\alpha z_i^2 -\sum\limits_{i=1}^n\sum\limits_{j=1}^n\frac{b_{ij}}{4}\left[\left(\frac{\sqrt{y_j}}{y_i\sqrt{y_i}}\right)^{\frac{1}{2}}z_i-\left(\frac{\sqrt{y_i}}{y_j\sqrt{y_j}}\right)^{\frac{1}{2}}z_j\right]^2\\
(b_{ij}\le 0,~i\neq j)&\ge\sum\limits_{i=1}^n\alpha z_i^2\ge 0.
\end{aligned}
\end{equation*}
The Hessian matrix of $f(y)$ is positive definite. Thus $f(y)$ is strictly convex over $\mathcal{S}$.
\end{proof}

Combine \cref{lem:equal} and \cref{lem:convex}, we conclude that \cref{equ:kkt} has a unique nonnegative eigenvector, and \cref{equ:different} has a unique nonnegative optimizer.
\end{remark}

Now, we can conclude that the nonpositive (or nonnegative) eigenvector corresponding to the smallest eigenvalue of the NEPv \cref{equ:kkt}, is exactly the global optimum of \cref{equ:different}.
\begin{theorem}\label{thm:global}
 Under \cref{ass:structure}, the minimization problem \cref{equ:different} obtains its global minimum if and only if the stationary point is the nonnegative (or nonpositive) eigenvector of \cref{equ:kkt}. Furthermore, when imposing the nonnegative constraint, the nonnegative eigenvector is the only global optimum.
\end{theorem}
\begin{proof}
The sufficiency is obvious. According to \cref{equ:proof_nonnegative}, there is always a nonnegative global optimum, which thus satisfies \cref{equ:kkt}. On the other hand, \cref{lem:eigen} says that the eigenpair $(\lambda,~x)$ with $x\ge 0$ is unique. Thus, if the stationary point $x$ is nonnegative (or nonpositive), it must be a global minimum.

Suppose $x^*$ is a global minimum, then according to \cref{equ:proof_nonnegative}, $|x^*|$ is also a global minimum. So there exist $\lambda$ and $\mu$, such that $(\lambda, x^*)$ and $(\mu, |x^*|)$ are corresponding eigenpairs, respectively. Then, we have
\begin{equation*}
\frac{\alpha}{2}\mathcal{A}(x^*)^4+(x^*)^TBx^*=
\frac{\alpha}{2}\mathcal{A}(|x^*|)^4+(|x^*|)^TB|x^*|,
\end{equation*}
and
\begin{equation*}
\begin{aligned}
\lambda &= \alpha\mathcal{A}(x^*)^4+(x^*)^TBx^*;\\
\mu &= \alpha\mathcal{A}(|x^*|)^4+(|x^*|)^TB|x^*|.
\end{aligned}
\end{equation*}
Thus, $\lambda=\mu$. According to \cref{lem:eigen}, $\lambda$ is real geometrically simple. That is, $x^*=|x^*| (or -|x^*|)$. The global optimum is nonnegative (or nonpositive).
\end{proof}
\begin{remark}
We can also prove that $H=\mathcal{A}x^2+B-\lambda_0 I\succeq 0$ as the proof of \cref{equ:proof2}, where $x$ is the corresponding eigenvector of $\lambda_0$. Then according to Theorem 2.1 in Zhang et al. \cite{zhang2019geometric}, we can directly come to the conclusion that $x$ is a global optimum and any other global minimum of \cref{equ:different} must belongs to the equivalence class $[\![x]\!]=\{y\in \mathbb{R}^n:|y_k|=|x_k|,~\forall k\in[n]\}$. Here, we obtain further that $y=x$ or $y=-x$.
\end{remark}
\begin{remark}
Before we move to the next section for the algorithms for solving the optimization problem in question, we note that Self-Consistent Field (SCF) is a widely used algorithm to solve the NEPv. And according to Cai et al. \cite[Theorem 3.1, Theorem 4.2]{cai2018eigenvector}, we can derive a rough sufficient condition for NEPv \cref{equ:kkt} to have a unique eigenvector $x^*$ corresponding to smallest eigenvalue and SCF to converge globally to $x^*$. That is, $0<\alpha<\frac{\lambda_{2}(B)-\lambda_{min}(B)}{3}$, where $\lambda_2(B)$ is the second smallest eigenvalue of $B$. Here, we only need $\alpha>0$ for NEPv with special $B$ to have a unique (up to a scalar) eigenvector corresponding to the smallest eigenvalue.
\end{remark}

\section{Convergence of algorithms to the global optimum}
\label{sec:algorithms}
Based on \cref{thm:global}, we derive that for any algorithm that can find a stationary point for \cref{equ:different}, we may check the global optimality by its sign. And if we impose the nonnegativity, those algorithms actually can find the global minimum.
\subsection{The regularized Newton method}\label{subsec:RN}
As an example, we briefly explain how to obtain the convergence to a positive global optimum for the regularized Newton (RN) method proposed by Wu et al. \cite{wu2017regularized}. Since the changes in the algorithm and the convergence proof are too trivial and the complete convergence analysis is complicated, we will not repeat them in detail and only point out the difference.

For the algorithm, we can achieve the nonnegativity by simply taking the absolute value of $X$ in the updating of $X^{k+1}$. For the convergence proof, we only make a tiny modification to it using notations that are consistent with \cite[Theorem 4.9]{wen2013adaptive}. That is, taking the absolute value of $Z_k$ in the updation of $X_{k+1}$, then $E(|Z_k|)\le E(Z_k)$. Thus we have $$E(X_k)-E(|Z_k|)\ge E(X_k)-E(Z_k)\ge\eta_1\cdot (-m_k(Z_k)),$$
where $E(X)$ is the total energy function (in the context, it is the objective function $\frac{\alpha}{2}\mathcal{A}x^4+x^{T}Bx$), $m_k(X)$ is an approximate Taylor expansion to $E(X)$, and $Z_k$ is the trial point computed from ${\min_{\|X\|_2=1}}m_k(X)$. The convergence to the stationary point is not affected according to \cite{wen2013adaptive}. Thus, the RN algorithm can converge to a global optimum for \cref{equ:different}.

Actually, Wu et al. \cite{wu2017regularized} showed numerically that the RN method without nonnegative restriction, can be applied to compute the asymmetric excited states provided that the initial data is chosen as an asymmetric function. It implies that the RN method may find a stationary point other than the optimum. With the simple absolute value operation, the RN method will always obtain the ground state whatever initial point it started with. We will give an example in the numerical section.

\subsection{Alternating direction method of multipliers}\label{subsec:ADMM}
The challenge to solve the spherical constraint problem comes from the nonlinear and nonconvex constraint. Penalty methods can be used to avoid handling the spherical constraint directly \cite{nocedal2006numerical}, while it usually suffers from slow convergence. Compared with the orthogonality preserving algorithm like mentioned above in \cref{subsec:RN}, alternating direction method of multipliers (ADMM) can be coded easily, and for the spherical constraint, its subproblem in the algorithm can be solved analytically. Osher et al. \cite{lai2014splitting} has already shown the efficiency of splitting method on the orthogonality constraint problem through numerical results without the theoretical guarantee. Taking the advantage of the structure of the problem considered in this paper, we obtain the convergence of ADMM without involving sophisticated manifold theories.

First, we rewrite \cref{equ:different} imposing the nonnegative constraint into the standard ADMM problem as follows:
\begin{equation}\label{equ:admm}
\left\{
\begin{array}{lrc}
\underset{x\in \mathbb{R}^n}{\min} \quad I_s(x)+f(y)\\
{\rm s.t.}\quad x=y,
\end{array}
\right.
\end{equation}
where $\mathcal{S} = \{x|\|x\|=1, x\ge 0\}$, $f(y)=\frac{\alpha}{2}\mathcal{A}y^4+y^TBy$. The augmented lagrangian of \cref{equ:admm} is,
\begin{equation}\label{equ:lagrangian}
\mathcal{L}_{\rho}(x,~y,~w)=I_s(x)+f(y)+w^T(x-y)+\frac{\rho}{2}\|x-y\|_2^2,
\end{equation}
for which, the iteration steps are:
\begin{equation}
\begin{array}{lrc}\label{equ:iter2}
x^{k+1}:=Proj_\mathcal{S}(y^{k}-\frac{w^{k}}{\rho});\\
y^{k+1}:=\underset{y}{\rm{arg}\min}(f(y)+w^{k^T}(x^{k+1}-y)+\frac{\rho}{2}\|x^{k+1}-y\|_2^2);\\
w^{k+1}:=w^{k}+\rho(x^{k+1}-y^{k+1}).
\end{array}
\end{equation}

In the rest part of this subsection, we give the convergence analysis of the standard ADMM for the special spherical constraint optimization problem considered here. Our analysis is based on the work of Wang et al. \cite{wang2019global}, which requires $f(y)$ in \cref{equ:admm} to be Lipschitz differentiable with constant $L_f$. However the Lipschitz condition is not satisfied in our problem, since $f(y)$ is a quartic function. The following lemma proves that the sequence $\{y_k\}$ generated by \cref{equ:iter2} is bounded without the need for $f(y)$ to be Lipschitz differentiable; thus $\nabla f(y)$ only need to have Lipschitz constant locally.
\begin{lemma}\label{lem:bound}
For any given initial point $y^0$ and $w^0$, if $\rho$ is sufficiently large, the sequence $\{(x^k,~y^k,~w^k)\}$ is bounded. In particular, for any $D>1$, if
\begin{equation}\label{equ:rhobound}
\rho\ge \max\{\frac{\|w^0\|-2D\lambda_{min}(B)}{D-1},~\frac{2D^3\alpha+2D(\lambda_{max}(B)-\lambda_{min}(B))}{D-1}\},
\end{equation}
 $\|y^k\|\le D~(\forall ~k\ge 1)$.
\end{lemma}
\begin{proof}
According to \cref{equ:iter2}, we obtain that
\begin{equation*}
\begin{aligned}
\|x^{k+1}\|&=1;\\
\nabla f(y^{k+1})-w^k-\rho(x^{k+1}-y^{k+1})&=0;\\
w^{k+1}&=w^k+\rho(x^{k+1}-y^{k+1}).
\end{aligned}
\end{equation*}
Hence, we have
\begin{equation}\label{equ:w}
\begin{aligned}
\nabla f(y^1)+\rho y^1&=w^0+\rho x^1;\\
\nabla f(y^k)&=w^k,~\forall k\ge 1;\\
\nabla f(y^{k+1})+\rho y^{k+1}&=\nabla f(y^k)+\rho x^{k+1},~\forall k\ge1.
\end{aligned}
\end{equation}
That is,
\begin{align}
2\alpha\mathcal{A}(y^1)^3+2By^1+\rho y^1&=w^0+\rho x^1;\label{equ:y1}\\
2\alpha\mathcal{A}(y^{k+1})^3+2By^{k+1}+\rho y^{k+1}&=2\alpha\mathcal{A}(y^{k})^3+2By^{k}+\rho x^{k+1},~\forall k\ge 1.\label{equ:yk}
\end{align}
Since $\alpha>0$, $B\succeq 0$, $2\alpha\mathcal{A}(y^{k+1})^2+2B+\rho I$ is positive definite for any $\rho>0$. For some $D>1$, we now give a bound for $\rho$, such that $\|y^k\|\le D$, for any $k\ge 1$.

First, according to \cref{equ:y1}, we have
\begin{equation*}
\begin{aligned}
\|y^1\|&=\|(2\alpha\mathcal{A}(y^1)^2+2B+\rho I)^{-1}(w^0+\rho x^1)\|\\
&\le\frac{\|w^0\|+\rho}{2\lambda_{min}(B)+\rho}.
\end{aligned}
\end{equation*}
It leads to that $\|y^1\|\le D$ when $\rho\ge\frac{\|w^0\|-2D\lambda_{min}(B)}{D-1}$.

Suppose $\|y^k\|\le D$, according to \cref{equ:yk}, we can induce the bound for $y^{k+1}$ similarly as follows,
\begin{equation*}
\begin{aligned}
\|y^{k+1}\|&=\|(2\alpha\mathcal{A}(y^{k+1})^2+2B+\rho I)^{-1}(2\alpha\mathcal{A}(y^{k})^3+2By^k+\rho x^{k+1})\|\\
&\le\frac{2\alpha\|y^k\|^3+2\lambda_{max}(B)\|y^k\|+\rho}{2\lambda_{min}(B)+\rho}\\
&\le\frac{2 D^3\alpha+2D\lambda_{max}(B)+\rho}{2\lambda_{min}(B)+\rho}.
\end{aligned}
\end{equation*}
Thus, $\|y^{k+1}\|\le D$ if $\rho\ge\frac{2D^3\alpha+2D(\lambda_{max}(B)-\lambda_{min}(B))}{D-1}$, for some $D>1$.
\end{proof}

For completeness, we recall and summarize some results of Wang et al.\cite{wang2019global} as the following lemma. And we also give a detailed proof of those results for our problem, based on the boundedness of $\{(x^k,~y^k,~w^k)\}$.
\begin{lemma}\label{lem:sum}
If $\rho$ satisfies \cref{equ:rhobound} and $\rho\ge2(L_f+1)$, where $L_f$ is the locally Lipschitz constant for $\nabla f(y)$ when $\|y\|\le D$ for some $D>1$, then for any $k\ge1$, it holds that
\begin{equation}\label{equ:subbound}
\|x^{k+1}-y^{k+1}\|\le \frac{L_f}{\rho}\|y^{k+1}-y^k\|,
\end{equation}
and
\begin{equation}\label{equ:decrease}
\mathcal{L}_{\rho}(x^k,~y^k,~w^k)-\mathcal{L}_{\rho}(x^{k+1},~y^{k+1},~w^{k+1})\ge \|y^{k+1}-y^{k}\|^2.
\end{equation}
\end{lemma}
\begin{proof}
According to \cref{lem:bound}, $\|y^k\|\le D~(\forall k\ge1)$ for some $D>1$ when $\rho$ satisfies \cref{equ:rhobound}. Then we can obtain a locally Lipschitz constant for $\nabla f(y)$, denoted as $L_f$. For any $k\ge 1$, we have
\begin{equation*}
\begin{aligned}
\|x^{k+1}-y^{k+1}\|&=\frac{\|w^{k+1}-w^k\|}{\rho}\\
(\text{according to }\cref{equ:w})&=\frac{\|\nabla f(y^{k+1})-\nabla f(y^k)\|}{\rho}\\
&\le \frac{L_f}{\rho}\|y^{k+1}-y^k\|.
\end{aligned}
\end{equation*}
For the sufficient descent of $\mathcal{L}_{\rho}(x^k,y^k,w^k)$, we first have
\[\mathcal{L}_{\rho}(x^k,y^k,w^k)\ge\mathcal{L}_{\rho}(x^{k+1},y^{k},w^{k}),\]
according to the optimality of $x^{k+1}$. Then,
\begin{equation*}
\begin{aligned}
&\mathcal{L}_{\rho}(x^k,~y^k,~w^k)-\mathcal{L}_{\rho}(x^{k+1},~y^{k+1},~w^{k+1})\\
\ge &\mathcal{L}_{\rho}(x^{k+1},~y^k,~w^k)-\mathcal{L}_{\rho}(x^{k+1},~y^{k+1},~w^{k+1})\\
=&f(y^k)-f(y^{k+1})+(w^{k+1})^T(y^{k+1}-y^k)+\frac{\rho}{2}\|y^{k+1}-y^k\|^2\\
&+(w^k-w^{k+1})^T(x^{k+1}-y^{k+1})\\
\ge& f(y^k)-f(y^{k+1})+(w^{k+1})^T(y^{k+1}-y^k)+\frac{\rho}{2}\|y^{k+1}-y^k\|^2\\
&-\frac{L_f^2}{\rho}\|y^{k+1}-y^k\|^2\\
\ge&-\frac{L_f}{2}\|y^{k+1}-y^{k}\|^2+\frac{\rho}{2}\|y^{k+1}-y^k\|^2-\frac{L_f^2}{\rho}\|y^{k+1}-y^k\|^2.
\end{aligned}
\end{equation*}
The first inequality comes from \cref{equ:w} and \cref{equ:subbound}. Thus, if $\rho\ge2(L_f+1)$, we obtain that
\begin{equation*}
\mathcal{L}_{\rho}(x^k,~y^k,~w^k)-\mathcal{L}_{\rho}(x^{k+1},~y^{k+1},~w^{k+1})\ge\|y^{k+1}-y^k\|^2,~\forall k\ge1.
\end{equation*}
\end{proof}
\begin{remark}
So far, \cref{lem:bound} and \cref{lem:sum} only need $B\succeq 0$ without other assumption on $B$ to be established. For brevity, in the numerical experiments, we can only update $x$ in the way $x^{k+1}:=(y^k-\frac{ w^k}{\rho})/\|y^k-\frac{ w^k}{\rho}\|$, and then check whether entries of the solution have the same sign. The boundedness of the generated sequence $\{(x^k,~y^k,~w^k)\}$ and \cref{lem:sum} can be established similarly. Then according to   Corollary 2 of Wang et al.\cite{wang2019global}, we can already come to that $\{(x^k,~y^k,~w^k)\}$ has at least one limit point. This limit point $(x^*,~y^*,~w^*)$ is a stationary point of the augmented Lagrangian $\mathcal{L}_{\rho}$. That is,
\begin{equation*}
\begin{aligned}
0&=x^*-y^*,\\
0&=\nabla f(y^*)-w^*,\\
0&=cx^*+w^*~\text{for some $c\in \mathbb{R}$}.
\end{aligned}
\end{equation*}
The last equality comes from the subdifferential of $I_{\hat{\mathcal{S}}}(x)~(\hat{\mathcal{S}}=\{x|\|x\|_2=1\})$ \cite{rockafellar2009variational}. So we can infer that $y^*$ is also a stationary point of \cref{equ:different}.
\end{remark}

For $\mathcal{S}=\{x|\|x\|_2=1,~x\ge 0\}$, $Proj_{\mathcal{S}}(y)$ also can be computed analytically.
\begin{lemma}\cite[Example 8.9]{bauschke2018projecting}\label{lem:projection}
$\mathcal{S} = \{x|\|x\|_2=1, x\ge 0\}$, and the projection onto $\mathcal{S}$ denoted by $Proj_{\mathcal{S}}(y):=\underset{x\in\mathcal{S}}{\rm{arg}\min}\|x-y\|$,
then
\begin{equation}\label{equ:proj}
Proj_{\mathcal{S}}(y) =
\left\{
\begin{array}{lrc}
\frac{Proj_{R_+}(y)}{\|Proj_{R_+}(y)\|},~if~\underset{i}{\max}~\{y_i\}>0\\
\{\sum\limits_i\alpha_ie_i|\sum\limits_i\alpha_i^2=1, y_i=0\},~if~\underset{i}{\max}~\{y_i\}=0\\
\{e_i|y_i=\underset{j}{\max}~\{y_j\}\},~if ~\underset{i}{\max}~\{y_i\}<0
\end{array}
\right.,
\end{equation}
where $e_i=(0,\cdots,\underset{i}{1},\cdots,0)^T$.
\end{lemma}

Now, we can reach the following theorem for convergence of the standard ADMM \cref{equ:iter2} to the global optimum.
\begin{theorem}\label{thm:ADMM}
 For some $D>1$ and any given initial point $(y^0,w^0)$, if $$\rho>\max\{\frac{\|w^0\|-2D\lambda_{min}(B)}{D-1},~\frac{2D^3\alpha+2D(\lambda_{max}(B)-\lambda_{min}(B))}{D-1},~2(L_f+1),~2\alpha+2\lambda_{max}(B)\},$$ where $L_f$ is the locally Lipschitz constant for $\nabla f(y)$ when $\|y\|\le D$, then the sequence $\{(x^k,~y^k,~w^k)\}$ generated by the standard ADMM \cref{equ:iter2}, will converge to $(x^*,~y^*,~w^*)$. $y^*$ is a global optimum of \cref{equ:different}.
\end{theorem}
\begin{proof}
Since $\{(x^k,~y^k,~w^k)\}$ is bounded, we have $\mathcal{L}_{\rho}(x^k,~y^k,~w^k)$ is lower bounded and $\sum_{k=1}^{\infty}\|y^{k+1}-y^k\|^2<\infty$ resulting from \cref{equ:decrease}. This implies that
\begin{equation}\label{equ:adjacent} \lim\limits_{k\rightarrow\infty}\|y^{k+1}-y^k\|=0.
\end{equation}
Then, we have
\[\lim\limits_{k\rightarrow\infty}\|x^{k}-y^{k}\|=0,\]
according to \eqref{equ:subbound}.
For any cluster point $(x^*,~y^*,~w^*)$ of the generated sequence, we have
\begin{align}
y^*=x^*\ge0;\label{equ:x=y}\\
\|y^*\|=\|x^*\|=1;\label{equ:norm}\\
w^*=\nabla f(y^*).\notag
\end{align}

Thus, if every cluster point $y^*$ is an eigenvector of \eqref{equ:kkt}, they are actually the same one, that is the unique global optimizer, according to \cref{lem:eigen}. And we can conclude that the entire sequence is convergent. Now, we only need to prove that $y^*$ is a nonnegative eigenvector for the NEPv \cref{equ:kkt}. We denote the subsequence that converges to $(x^*,~y^*,~w^*)$ as $\{(x^{k_j},~y^{k_j},~w^{k_j})\}$.

\textbf{\boldmath{Case 1, There exists a nonempty set $I\subset[n]$, $\overline{I}=[n]\backslash I$, such that $y^*_i>0~(i\in \overline{I})$ and $y^*_i=0~(i\in I)$.}\boldmath}\\
 Let $z=y^*-\frac{\nabla f(y^*)}{\rho}$, then Assumption \ref{ass:structure} results in the fact that there must be a $z_i>0$ for some $i\in I$. Otherwise,  $$0\ge z_i=-\frac{2\sum_{j\neq i}b_{ij}y^*_j}{\rho}=-\frac{2\sum_{j\in \overline{I}}b_{ij}y^*_j}{\rho},~\forall i\in I.$$
 This leads to $$b_{ij}=0,~ \forall j\notin I,$$ which contradicts the irreducibility assumption. Thus $x^*_i>0$ for some $i\in I$ correspondingly, which contradicts $x^*_i=y^*_i=0$. We come to the conclusion that $y^*>0$.

\textbf{Case 2, \boldmath${x^*=y^*>0}$\boldmath.} There exists $j_0$, $x^{k_j}>0~(\forall j>j_0)$. According to \cref{lem:projection},
$$x^{k_j}=\frac{y^{k_j-1}-\frac{\nabla f(y^{k_j-1})}{\rho}}{\|y^{k_j-1}-\frac{\nabla f(y^{k_j-1})}{\rho}\|};$$ or $$x^{k_j}\in\{\sum\limits_i\alpha_ie_i|\sum\limits_i\alpha_i^2=1,~\alpha_i\neq 0,~(y^{k_j-1}-\frac{\nabla f(y^{k_j-1})}{\rho})_i=0,~\forall~i\in[n]\}.$$

However the second kind of $x^{k_j}$ is impossible. According to \cref{equ:norm},
$$\lim_{j\rightarrow \infty}\|y^{k_j}\|=1.$$ On the other hand,
\begin{equation*}
\begin{aligned}
\|\nabla f(y^{k_j})\|&=\|2\alpha \mathcal{A}(y^{k_j})^3+2By^{k_j}\|\le 2\alpha\|y^{k_j}\|^3+2\lambda_{max}(B)\|y^{k_j}\|.
\end{aligned}
\end{equation*}
Thus, if $\rho>2\alpha+2\lambda_{max}(B)$, for sufficiently large $j$, $\|y^{k_j}\|>\frac{\|\nabla f(y^{k_j})\|}{\rho}$, which implies that
$y^{k_j}-\frac{\nabla f(y^{k_j})}{\rho}\neq 0$. Combined with \eqref{equ:adjacent}, we obtain that for sufficient large $j$,
$$y^{k_j-1}-\frac{\nabla f(y^{k_j-1})}{\rho}\neq 0.$$ Thus,
\begin{equation*}
y^*=x^*=\underset{j\rightarrow \infty}{\lim}x^{k_j}=\underset{j\rightarrow \infty}{\lim}\frac{y^{k_j-1}-\frac{\nabla f(y^{k_j-1})}{\rho}}{\|y^{k_j-1}-\frac{\nabla f(y^{k_j-1})}{\rho}\|}=\frac{y^*-\frac{\nabla f(y^*)}{\rho}}{\|y^*-\frac{\nabla f(y^*)}{\rho}\|}.
\end{equation*}
That is,
\begin{equation*}
\begin{aligned}
\nabla f(y^*)&=\rho(1-\|y^*-\frac{\nabla f(y^*)}{\rho}\|)y^*;\\ \|y^*\|&=1.
\end{aligned}
\end{equation*}
It is obvious that $y^*$ is the positive eigenvector of \cref{equ:kkt}.
\end{proof}

\subsection{Inexact ADMM}\label{subsec:inexact}
Before we move on to the numerical section, the convergence of the inexact version of \eqref{equ:iter2} is discussed, since there is no analytical solution for the subproblem of $y$.
\begin{theorem}\label{thm:inexact}
Suppose that there is a nonnegative non-increasing sequence $\{\epsilon_k\}$ such that $\sum_k\epsilon_k<+\infty$. The solution of the subproblem of $y$ in \eqref{equ:iter2} satisfies
\[\|\nabla_y \mathcal{L}(x^{k+1},~y^{k+1},~w^k)\|=\|\nabla f(y^{k+1})-w^k-\rho(x^{k+1}-y^{k+1})\|\le\epsilon_k.\]
Then, the convergence results in \cref{thm:ADMM} for this kind of inexact ADMM still hold.
\end{theorem}
\begin{proof}
The analysis scheme is analogous to that for the exact version. So we just simply note the differences in the proof above.

First, for the boundedness of $\{(x^k,~y^k,~w^k)\}$, \eqref{equ:w} in \cref{lem:bound} turns out to be
\begin{equation*}
\begin{aligned}
\|\nabla f(y^1)+\rho y^1\|&\le\|w^0+\rho x^1\|+\epsilon_1;\\
\|\nabla f(y^k)-w^k\|&\le \epsilon_k,~\forall k\ge 1;\\
\|\nabla f(y^{k+1})+\rho y^{k+1}\|&\le\|\nabla f(y^k)+\rho x^{k+1}\|+\epsilon_{k+1},~\forall k\ge1.
\end{aligned}
\end{equation*}
It is obvious that the remaining proving process of \cref{lem:bound} will not be influenced. Thus, the sequences generated by the inexact ADMM is still bounded with proper $\rho$.

Then analogous to \cref{lem:sum}, we have
\begin{equation*}
\begin{aligned}
\|x^{k+1}-y^{k+1}\|&=\frac{\|w^{k+1}-w^k\|}{\rho}\\
&=\frac{\|(w^{k+1}-\nabla f(y^{k+1}))-(w^k-\nabla f(y^k))+\nabla f(y^{k+1})-\nabla f(y^k)\|}{\rho}\\
&\le \frac{L_f}{\rho}\|y^{k+1}-y^k\|+\frac{2\epsilon_k}{\rho},
\end{aligned}
\end{equation*}
and
\begin{equation*}
\begin{aligned}
&\mathcal{L}_{\rho}(x^k,~y^k,~w^k)-\mathcal{L}_{\rho}(x^{k+1},~y^{k+1},~w^{k+1})\\
\ge&f(y^k)-f(y^{k+1})+(w^{k+1})^T(y^{k+1}-y^k)+\frac{\rho}{2}\|y^{k+1}-y^k\|^2\\
&+(w^k-w^{k+1})^T(x^{k+1}-y^{k+1})\\ =&f(y^k)-f(y^{k+1})+(\nabla f(y^{k+1}))^T(y^{k+1}-y^k)+\frac{\rho}{2}\|y^{k+1}-y^k\|^2\\
&-\rho\|x^{k+1}-y^{k+1}\|^2+(w^{k+1}-\nabla f(y^{k+1}))^T(y^{k+1}-y^k)\\
\ge&-\frac{L_f}{2}\|y^{k+1}-y^{k}\|^2+\frac{\rho}{2}\|y^{k+1}-y^k\|^2-\frac{L_f^2}{\rho}\|y^{k+1}-y^k\|^2\\
&-\epsilon_{k+1}\|y^{k+1}-y^k\|-\frac{1}{\rho}(4\epsilon_k^2+4\epsilon_kL_f\|y^{k+1}-y^{k}\|).\\
\end{aligned}
\end{equation*}
Since $y^k$ is bounded, for all $k\ge 1$ and $\sum_k\epsilon_k<+\infty$, we obtain that with proper $\rho$
\begin{equation*}
\mathcal{L}_{\rho}(x^k,~y^k,~w^k)-\mathcal{L}_{\rho}(x^{k+1},~y^{k+1},~w^{k+1})\ge C\|y^{k+1}-y^k\|^2-\eta_k,
\end{equation*}
where constant $C>0$,
\[\eta_k=\epsilon_{k+1}\|y^{k+1}-y^k\|+\frac{1}{\rho}(4\epsilon_k^2+4\epsilon_kL_f\|y^{k+1}-y^{k}\|),\]
$\sum_k\eta_k<+\infty$.
Thus, we can still derive that
\[\lim\limits_{k\rightarrow\infty}\|y^{k+1}-y^{k}\|=0,~\lim\limits_{k\rightarrow\infty}\|x^{k}-y^{k}\|=0.\]
The rest of the proof is almost the same as \cref{thm:ADMM}.
\end{proof}
\begin{remark}
In the next section, we use the Newton method to solve the subproblem of $y$, which can generate a solution that satisfies the condition in \cref{thm:inexact}.
\end{remark}

\section{Numerical results}
\label{sec:numerical}

In this section, by application on the non-rotating BEC problem, we explain our theories about global optimum and explore the convergence of the RN method and inexact ADMM for solving this special nonconvex optimization problem with numerical experiments.

The energy functional minimization problem of non-rotating BEC is defined as
\begin{equation*}\label{equ:energy}
\left\{
\begin{array}{lrc}
\min \quad E(\phi(\textbf{x})):= \int_{R^d}[\frac{1}{2}|\nabla\phi(\textbf{x})|^2+V(\textbf{x})|\phi(\textbf{x})|^2+\frac{\beta}{2}|\phi(\textbf{x})|^4]d\textbf{x} \\
{\rm s.t.}\quad \int_{\mathbb{R}^d}|\phi(\textbf{x})|^2d\textbf{x}=1, ~E(\phi)<\infty.
\end{array}
\right.
\end{equation*}
where $\textbf{x}\in \mathbb{R}^d$ is the spatial coordinate vector, $V(\textbf{x})$ is an external trapping potential, and the given constant $\beta$ is the dimensionless interaction coefficient, see \cite{bao2013mathematical}. The minimizer $\phi^*(\textbf{x})$ is defined as the ground state. We only consider $\beta>0$ in this paper. In most applications of BEC, the harmonic potential is used \cite{bao2004computing,bao2003ground}.
\begin{equation*}
V(\textbf{x})=\frac{1}{2}
\left\{
\begin{array}{lrc}
\gamma_x^2x^2,& d=1,\\
\gamma_x^2x^2+\gamma_y^2y^2, & d=2,\\
\gamma_x^2x^2+\gamma_y^2y^2+\gamma_z^2z^2, & d=3,
\end{array}
\right.
\end{equation*}
where $\gamma_x$, $\gamma_y$ and $\gamma_z$ are three given positive constants. $V(\textbf{x})$ can have other forms, which turn out to be positive diagonal matrix after the finite difference. So we only take the harmonic potential as the example in our numerical experiments.
Using the finite difference, we can reformulate the BEC problem as \cref{equ:different}. Unless otherwise specified, we take $\gamma_x=\gamma_y=\gamma_z=1$, and the space domain $D=[0,1],~[0,1]\times[0,1],~[0,1]\times[0,1]\times[0,1]$, for $d=1,~d=2,~d=3$, respectively.

Through the finite difference method with difference step as $h=\frac{1}{n}$ and dividing the $D$ evenly along each direction, the coefficient $\alpha$ in \cref{equ:different} will be $\beta n,~\beta n^2,~\beta n^3$ accordingly. And $B$ is a symmetric positive definite sparse matrix satisfying \cref{ass:structure}. With the division getting finer, that is, $n$ going large, the scale of discretization problem increases rapidly. We refer the reader to \cite{bao2013optimal} for the convergence of this finite difference discretization problem to the original energy functional optimization problem.

We implement all the following algorithms in MATLAB (Realease 2016b) and perform them on a Lenovo laptop with an Intel(R) Core(TM) Processor with access to 8GB of RAM.
\subsection{Implementation details}
For the NR method, the parameters are set as Wu et al. \cite{wu2017regularized}. The stopping criterion is $$\|x^{k+1}-x^{k}\|_{\infty}\le 10^{-6}.$$
For ADMM, although we have given a rough bound for the parameter $\rho$ in \cref{thm:ADMM}, it seems too large for a satisfactory convergence in practice and is sensitive to the performance. We choose the relatively good one after tuning and present different results of $\rho$ in \cref{subsec:rho}. The solver for the convex subproblem in ADMM is the Newton method, and we use Gauss-Seidel method to get the descent direction. The stopping criterion for ADMM is $$\|x^{k+1}-y^{k+1}\|\le \epsilon^{pri},~\|\rho(x^{k+1}-x^{k})\|\le \epsilon^{dual},$$
where
\begin{equation*}
\begin{aligned}
\epsilon^{pri}&=\sqrt{n}\epsilon^{abs}+\epsilon^{rel}\max\{\|x^{k+1}\|,\|y^{k+1}\|\},\\
\epsilon^{dual}&=\sqrt{n}\epsilon^{abs}+\epsilon^{rel}\|\rho w^{k+1}\|,
\end{aligned}
\end{equation*}
with $\epsilon^{abs}=\epsilon^{rel}=10^{-6}$.

In the following, $N$ stands for the number of split points along each direction, including two endpoints.
\subsection{Solving the SDP problem}
According to \cref{rem:sdp,lem:tight}, the SDP relaxation problem \cref{equ:sdp1} is a convex problem and equivalent to the origin nonconvex problem \cref{equ:different}. It can be computed directly by CVX or QSDPNAL \cite{li2018qsdpnal}. So we can regard the value computed from the SDP relaxation problem as the global optimization value of \cref{equ:different} to validate the convergence to a global optimum of the RN method and ADMM. In \cref{tab:sdpvalue} we take the QSDPNAL to solve the one-dimensional discretized BEC problem as an example with $\beta=0.5$, which will stop with the relative KKT residual of \cref{equ:sdp1} smaller that $10^{-8}$. \cref{fig:sdptime} illustrates the increase in computation time as the problem scale becomes larger.
\makeatletter
\newcommand\figcaption{\def\@captype{figure}\caption}
\newcommand\tabcaption{\def\@captype{table}\caption}
\makeatother

\begin{figure}[H]
	\centering
		\begin{minipage}{0.4\textwidth}
			\centering
            \tabcaption{Solving the SDP relaxation problem by QSDPNAL in one-dimensional case.}
            \label{tab:sdpvalue}
            \begin{tabular}{|l|lll|}
            \hline
            N-1   & 50     & 100     & 200     \\ \hline
            obj & 5.4477 & 5.4489  & 5.4492  \\ \hline
            $\lambda_0$ & 5.8198 & 5.8210 & 5.8214 \\ \hline
            cpu(s) & 0.7160 & 0.4440  & 0.8470  \\ \hline
            \multicolumn{4}{|l|}{}           \\ \hline
            N-1   & 500    & 1000    & 1500    \\ \hline
            obj & 5.4493 & 5.4499  & 5.4502  \\ \hline
            $\lambda_0$ & 5.8214 & 5.8215 & 5.8215 \\ \hline
            cpu(s) & 6.4510 & 29.5710 & 91.3060 \\ \hline
            \end{tabular}
		\end{minipage}
	\hspace{0.5in}
		\begin{minipage}[]{0.4\textwidth}
			\centering
            \includegraphics[width=6cm]{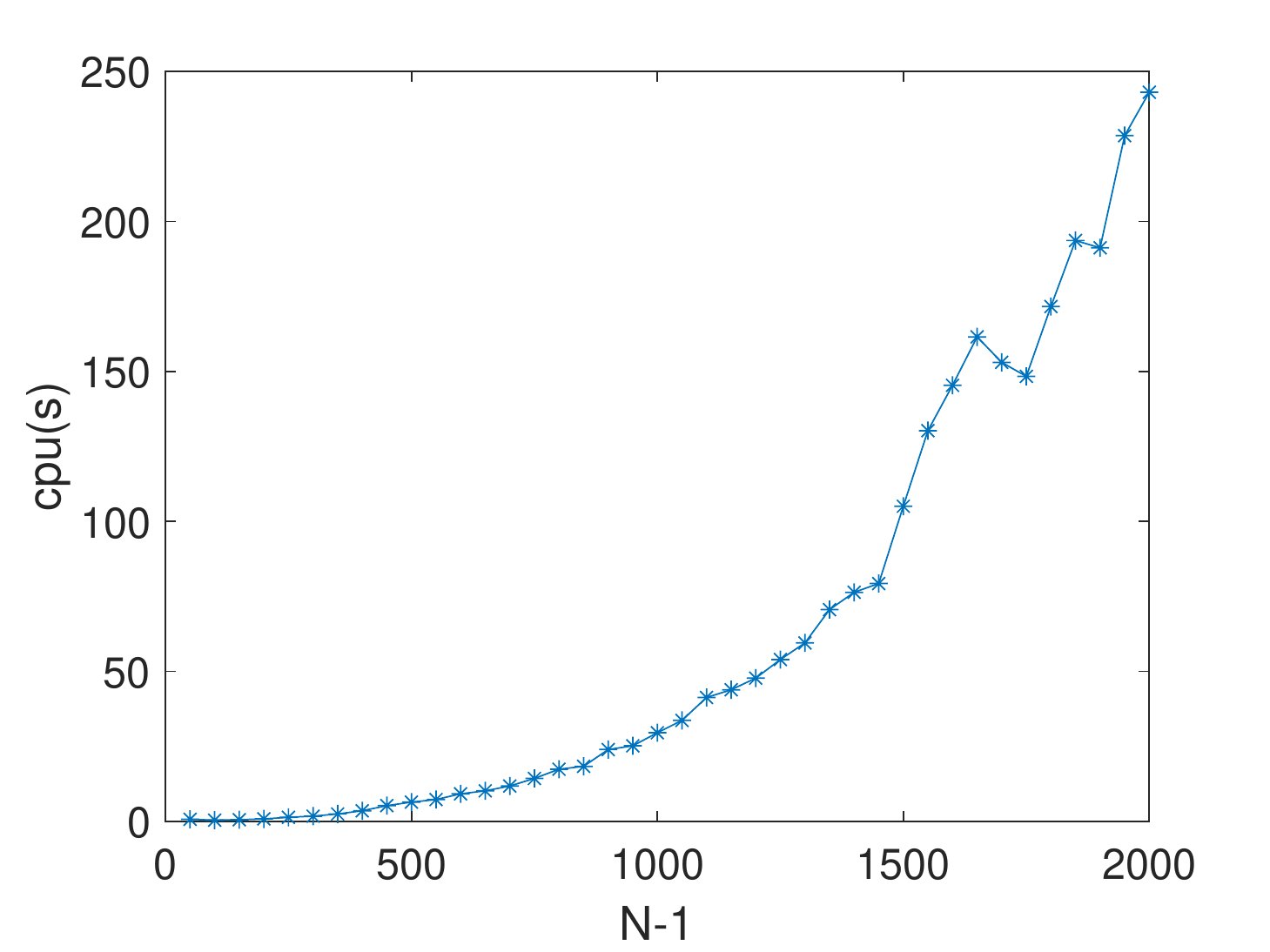}
            \figcaption{The cpu time via the problem scale.}
            \label{fig:sdptime}
		\end{minipage}%
\end{figure}

Although we have taken the advantage of the structure of \cref{equ:different} to just relax it into a quadratic SDP problem, instead of relaxing it further into the SDP problem with linear objective function. The SDP problem still becomes too large to be solved efficiently for two and three dimensional discretized BEC problems, see \cref{tab:spdlarge}.
\begin{table}[H]\label{tab:spdlarge}
\caption{Solving the SDP problem by QSDPNAL in three-dimensional case. The superscript '*' means that the solver stops without finding the optimum.}
\centering
\begin{tabular}{|cl|l|l|l|l|}
\hline
\multicolumn{1}{|l}{}                      &        & N=9     & N=17     & N=33    & N=65     \\ \hline
\multicolumn{1}{|c|}{\multirow{2}{*}{d=2}} & cpu(s) & 0.6720  & 1.2730   & 31.7290 & 726.4940 \\ \cline{2-6}
\multicolumn{1}{|c|}{}                     & obj    & 10.5802 & 10.6755  & 10.6994 & {$10.7069^*$}  \\ \hline
\multicolumn{1}{|c|}{\multirow{2}{*}{d=3}} & cpu(s) & 3.6185  & 2.5075e3 &\multicolumn{1}{c|}{--} & \multicolumn{1}{c|}{--}  \\ \cline{2-6}
\multicolumn{1}{|c|}{}                     & obj    & 15.5864 & $16.0012^*$  &\multicolumn{1}{c|}{--} & \multicolumn{1}{c|}{--}    \\ \hline
\end{tabular}
\end{table}

We also observed from the numerical results that when $\beta=0.5$,  the smallest eigenvalue and optimal value tend to increase as $N$ becomes large. Similar numerical behavior can be seen in Yang  et al. \cite{yang2019numerical}.
\begin{conjecture}
The smallest eigenvalue and the optimal value of the discretized problem are monotone nondecreasing with $N$ when $N$ is large enough, and converge to those of the original problem.
\end{conjecture}
The theoretical analysis for this phenomenon utilizing linear algebraic theory might worth further discussion.
\subsection{Choice of $\rho$}\label{subsec:rho}
In this subsection, we discuss the choice of $\rho$ for ADMM. The Riemannian gradient $nrmG=\|\nabla f(x)-(x^T\nabla f(x))x\|$ via the outer iteration of ADMM are plotted in \cref{fig:rho}. We only show parts of $\rho$ we have tried in the two dimensional case. If the $\rho$ is too small, the algorithm might divergent, while the ones too large will lead to slow convergence.
\begin{figure}[H]
\centering
\includegraphics[width=10cm]{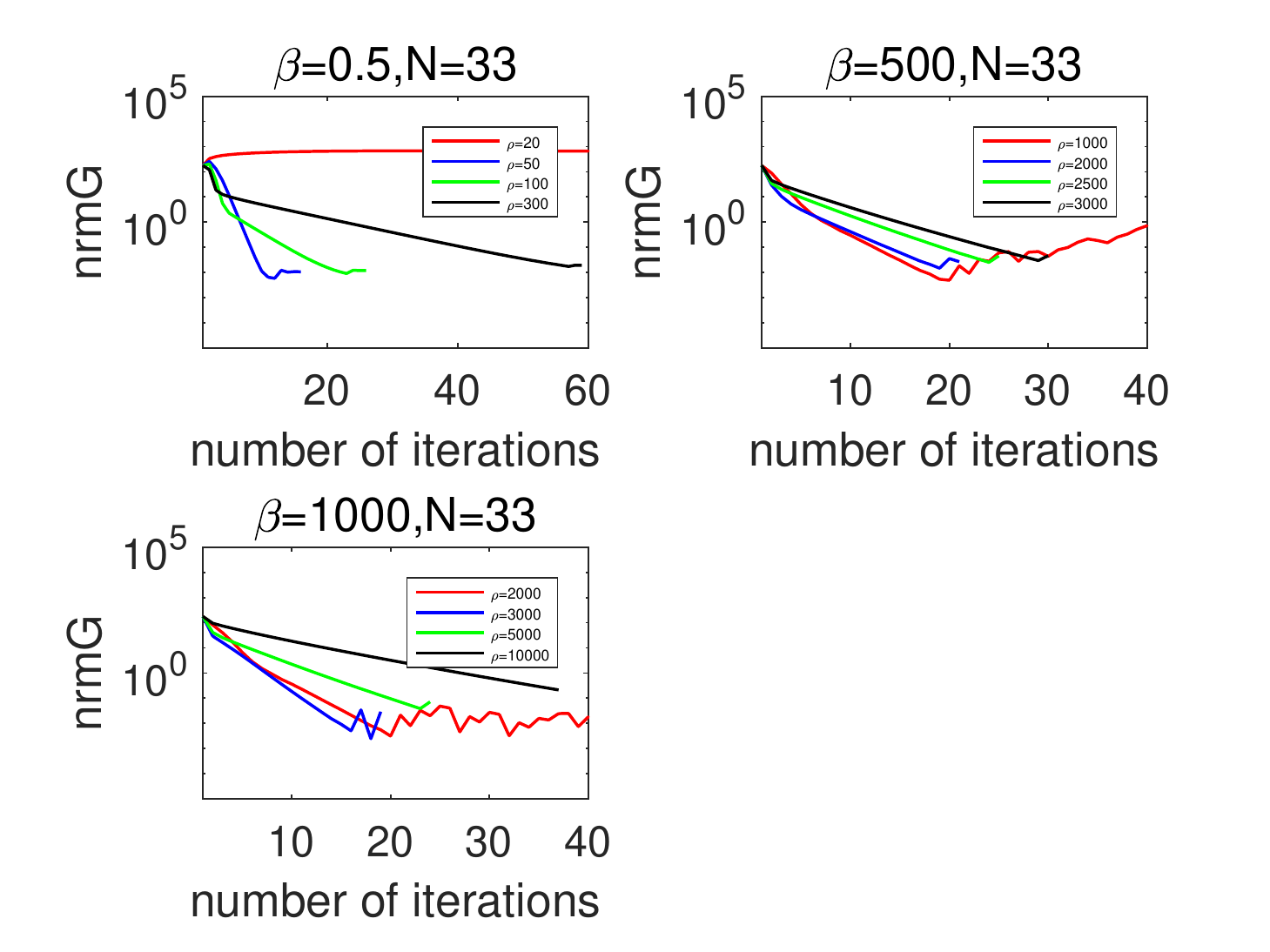}
\caption{Convergence of nrmG for ADMM with different $\rho$, in the two dimensional case with $N=33$.}
\label{fig:rho}
\end{figure}
However, we also found in our experiments, the average running time for each iteration is decreasing as $\rho$ increases, see \cref{fig:rho_time}. In general, the larger the interaction coefficient $\beta$ is, the larger an appropriate $\rho$ will be needed, which is consistent with the condition for $\rho$ in \cref{thm:ADMM}.
\begin{figure}[H]
\centering
\includegraphics[width=6cm]{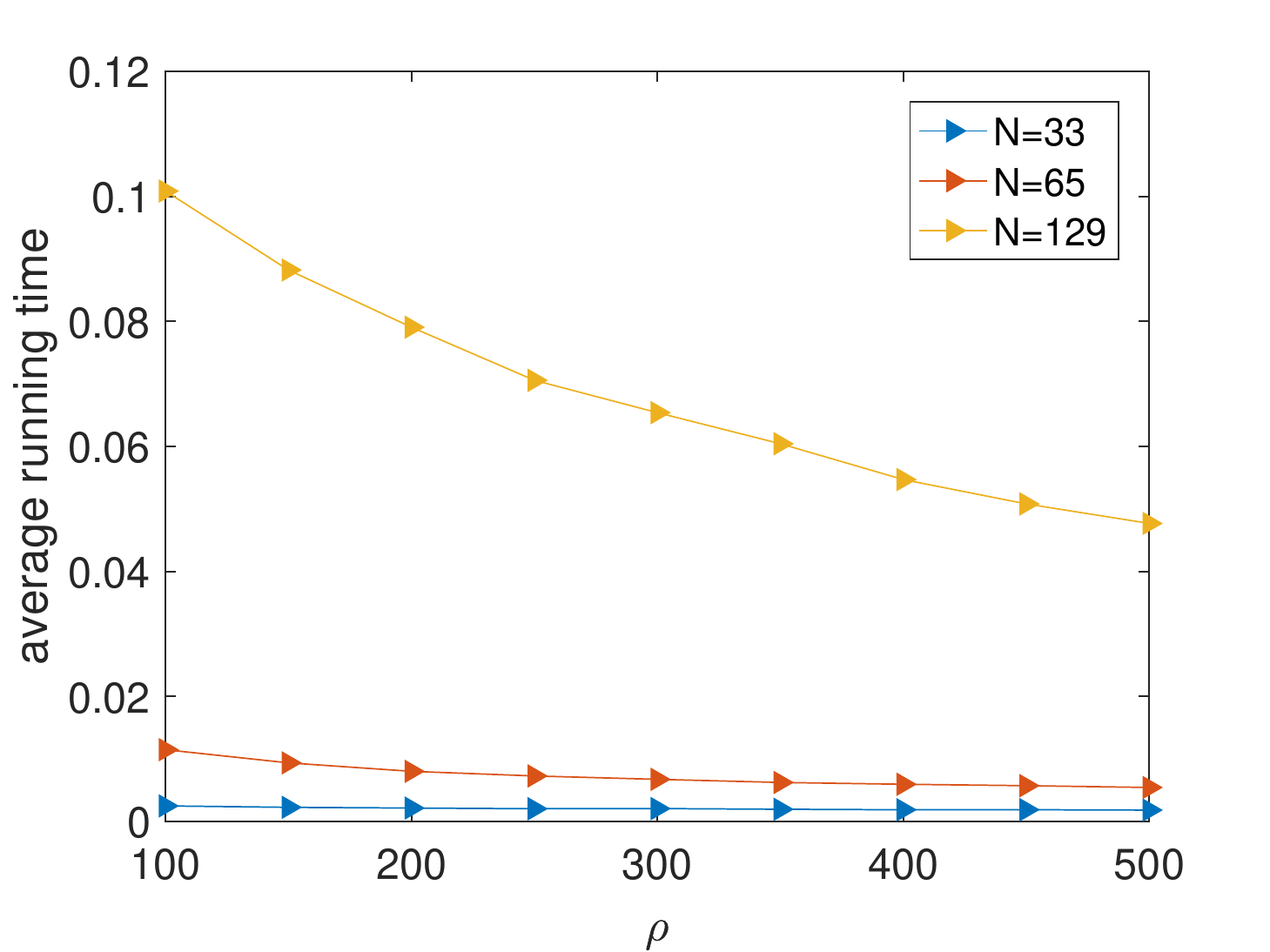}
\caption{The average running time via $\rho$ when $\beta$=0.5 in the two dimensional case.}
\label{fig:rho_time}
\end{figure}
\subsection{Convergence to the global optimum}
To validate our theorem about the global optimum and the convergence of the RN method and ADMM for the problem considered here, we first solve the BEC problem with SDP relaxation method in the case of $\beta=0.5$.

Not surprisingly, the optimal values of the RN method and ADMM solving the discretized BEC problem are the same as the SDP relaxation method. And the entries of the optimizers found by the RN method and ADMM  always have the same sign, which correspond to the smallest eigenvalue of \cref{equ:kkt}, see \cref{tab:ADMM}. Here, the parameter $\rho$ of ADMM is fixed as 100.
\begin{table}[H]
{\footnotesize
{
 \caption{In the first column, $d$ represents the space dimension. The second column shows the smallest eigenvalue corresponding to the solution computed by ADMM. The third and fourth columns are objective values of ADMM and the RN method, respectively. The sdp obj is the optimal value computed by SDP relaxation method. The last column checks whether the solutions obtained by ADMM and the RN method have all entries with the same sign. Y stands for they indeed have the same sign.}
 \label{tab:ADMM}
 }
 \begin{center}
 \begin{tabular}{cccccc}
 \toprule
  & $\lambda$ & ADMM obj & RN obj & sdp obj & sign\\
 \midrule
 d=1,N=257&5.8214&5.4492&5.4492&5.4492&Y\\
 d=1,N=513&5.8214&5.4493&5.4493&5.4493&Y\\
 d=2,N=9&11.1280&10.5802&10.5802&10.5802&Y\\
 d=2,N=17&11.2246&10.6755&10.6755&10.6755&Y\\
 d=3,N=5&16.0687&15.2886&15.2886&15.2886&Y\\
 d=3,N=9&16.6514&15.8564&15.8564&15.8564&Y\\
 \bottomrule
 \end{tabular}
 \end{center}
 }
 \end{table}

Yang  et al. \cite{yang2019numerical} computed all eigenpairs of NEPv\cref{equ:kkt} for the non-rotating BEC problem in small scale, which also displayed numerically that only the smallest eigenvalue has positive eigenvector. The other eigenvalues only have eigenvectors with mixed signs.

In \cref{fig:ground state}, it shows some examples of the discretized ground state computed by ADMM.
\begin{figure}[H]
\centering
\subfloat{\label{fig:a}}\includegraphics[width=6cm]{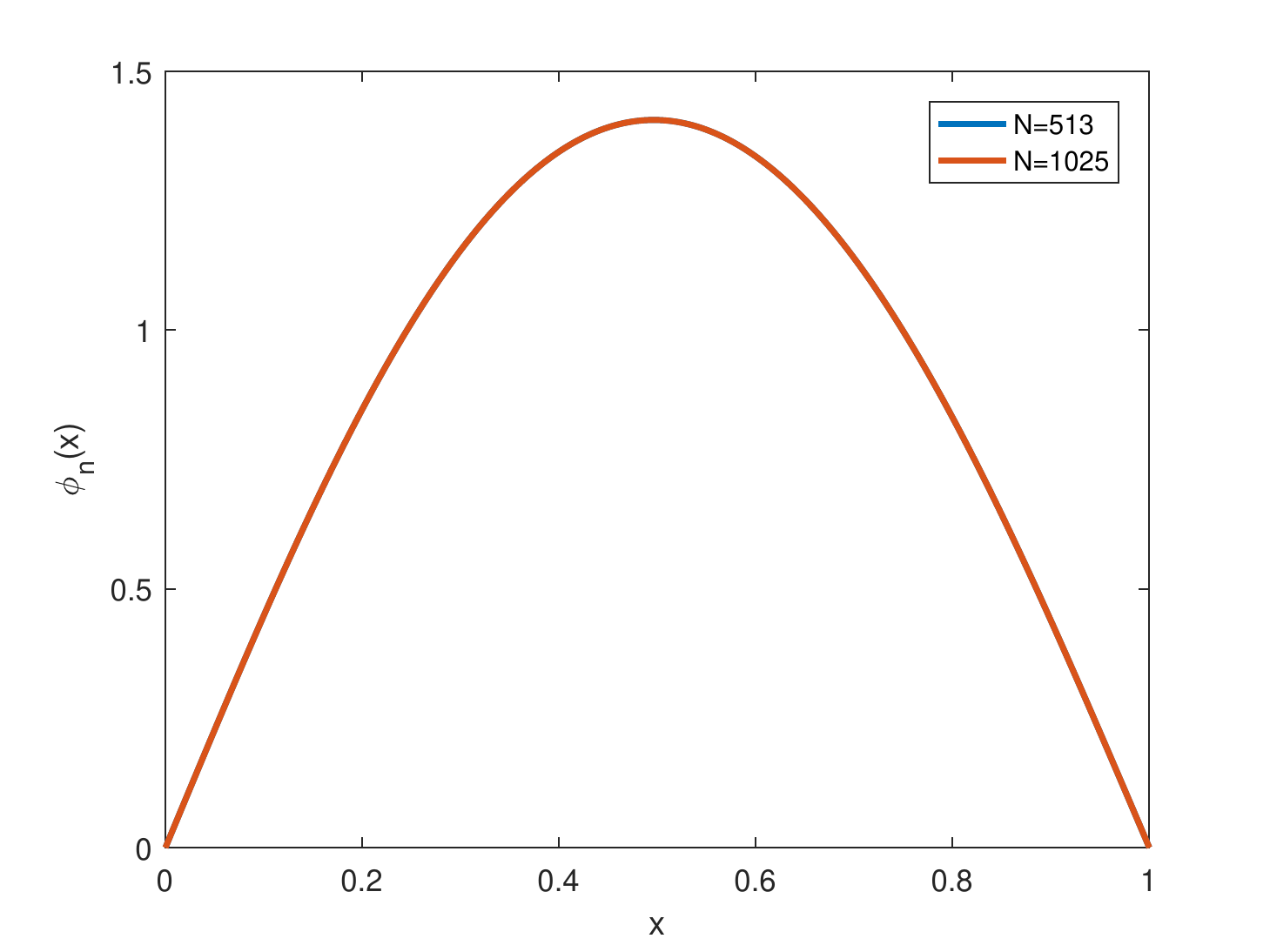}
\subfloat{\label{fig:b}}\includegraphics[width=6cm]{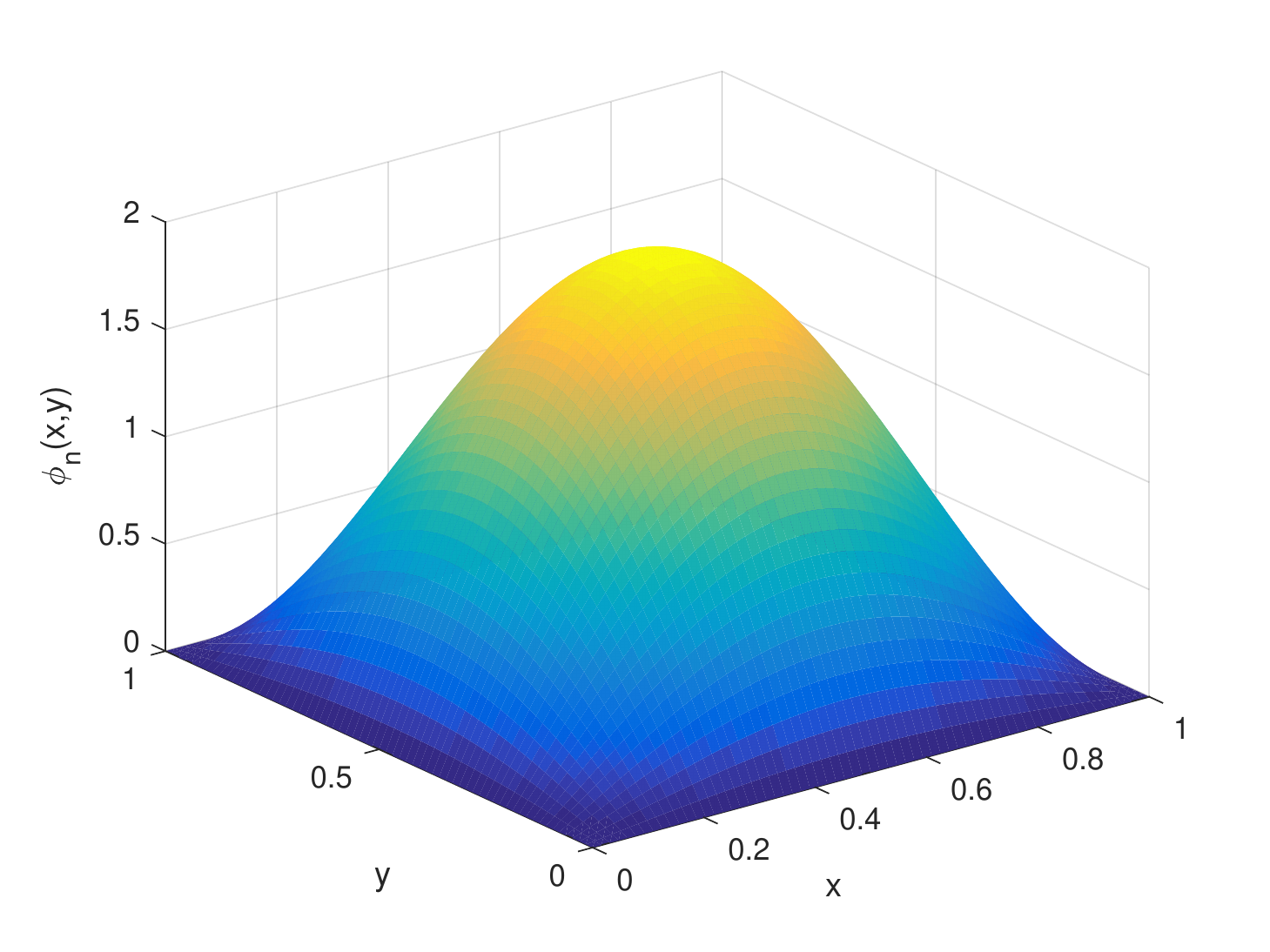}
\caption{The discretized ground state computed by ADMM. The left one is for one dimensional space with 513, 1025 total split points, respectively. The right one is the discretized ground state obtained for two dimensional space with 65 split points in each direction.}
\label{fig:ground state}
\end{figure}

In regard of the necessity of the nonnegative restriction for the RN method to converge to the global optimum, we take the special initial point given by Wu et al. \cite{wu2017regularized} for computing asymmetric excited states. See \cref{exa:excited}.
\begin{example}\label{exa:excited}
Let the domain $D=[-1,1]\times[-1,1]$, $\beta=0.5$, $N=2^4+1$. If the initial data is chosen as $\phi_0(x,y)=\frac{\sqrt{2}x}{\pi^{1/2}}e^{-(x^2+y^2)/2}$, it obtains a stationary point with the objective value 10.1652. Taking the absolute of $X$ in each $X-upstation$ step of the RN method, we always obtain the optimal value as 2.7307, which is the same as the optimal value computed by the convex SDP relaxation problem.
\end{example}
\subsection{Comparison between the RN method and ADMM}
We further compare ADMM with the RN method for the two-dimensional case with $\beta=0.5,~500,~1000$. We refine the mesh from $(2^4+1)\times(2^4+1)$ to $(2^7+1)\times(2^7+1)$  and use the solution computed for coarse mesh as a heuristic initial point for the next refined mesh. Part of the results is presented in the \cref{tab:compare}. The "total iter" columns only count the iteration number for the current $N$, while the "cpu(s)" are the total time from the coarsest grid. When $\beta=0.5$, $\rho=2000$ for $N=2^7+1$ and $\rho=100$ otherwise; when $\beta=500$, $\rho=1000,~1500,~2000,~2000$ for $N=2^4+1$ to $N=2^7+1$; when $\beta=1000$, $\rho=3000$.
\begin{table}[H]
{\footnotesize
    \caption{Comparison between ADMM and the RN method. The columns of total iter show the number of iterations. For RN method, it includes the iteration of feasible method; for ADMM, it includes the inner iteration for solving the unconstrained convex subproblem, and the numbers in brackets stand for the outer ADMM iteration. The columns of cpu are the cumulative time from the coarsest mesh. The nrmG column is the norm of Riemannian gradient $\|\nabla f(x)-(x^T\nabla f(x))x\|$, where $x$ is the computed optimizer, $f(x)$ is the objective function.}
    \label{tab:compare}
  \begin{center}
  \begin{threeparttable}
    \begin{tabular}{ccccccccc}
    \toprule
    \multirow{2}{*}{N}&\multicolumn{4}{c}{RN }&\multicolumn{4}{c}{ADMM}\cr
    \cmidrule(lr){2-5} \cmidrule(lr){6-9}
    &total iter&cpu(s)&obj val&nrmG&total iter&cpu(s)&obj val&nrmG\cr
    \midrule
    \multicolumn{9}{c}{$\beta=0.5$}\cr
    \midrule
    33&67&0.0922&10.6994&4.3e-4&39(19)&0.3975&10.6994&5.1e-3\cr
    65&63&0.2279&10.7054&2.7e-4&37(18)&1.0278&10.7054&6.0e-3\cr
    129&38&0.4475&10.7069&2.0e-3&13(8)&1.5123&10.7069&8.7e-2\cr
    \midrule
    \multicolumn{9}{c}{$\beta=500$}\cr
    \midrule
    17&13&0.0149&315.9526&1.2e-4&59(24)&0.0688&315.9526&2.8e-4\cr
    33&29&0.0483&313.6436&7.9e-5&47(20)&0.1597&313.6436&2.7e-3\cr
    129&67&0.6036&314.8709&1.1e-3&50(22)&1.8588&314.8709&1.7e-2\cr
    \midrule
    \multicolumn{9}{c}{$\beta=1000$}\cr
    \midrule
    17&9&0.1109&601.7806&1.7e-4&46(18)&0.0620&601.7806&3.4e-4\cr
    33&23&0.1432&587.4256&4.4e-4&48(20)&0.1510&587.4526&6.8e-3\cr
    65&55&0.3379&588.7273&2.4e-4&48(20)&0.3965&588.7273&8.4e-3\cr
    129&54&0.6935&589.3950&4.4e-3&47(19)&1.5122&589.3950&3.0e-1\cr
    \bottomrule
    \end{tabular}
    \end{threeparttable}
    \end{center}
}
\end{table}

\cref{tab:3d} presents results for three-dimensional case with $\beta = 0.5$. We refine the mesh from $(2^4+1)\times(2^4+1)\times(2^4+1)$ to $(2^7+1)\times(2^7+1)\times(2^7+1)$. For $N=17$, $\rho = 100$; $N=33$, $\rho = 200$; $N=65$, $\rho=900$; $N=129$, $\rho=30000$.
\begin{table}[H]
{\footnotesize
    \caption{Comparison between ADMM and the RN method in three-dimensional case with $\beta=0.5$.}
    \label{tab:3d}
  \begin{center}
  \begin{threeparttable}
    \begin{tabular}{ccccccccc}
    \toprule
    \multirow{2}{*}{N}&\multicolumn{4}{c}{RN }&\multicolumn{4}{c}{ADMM}\cr
    \cmidrule(lr){2-5} \cmidrule(lr){6-9}
    &total iter&cpu(s)&obj val&nrmG&total iter&cpu(s)&obj val&nrmG\cr
    \midrule
    17&64&0.1902&16.0005&1.3e-4&140(29)&1.0733&16.0005&3.1e-3\cr
    33&90&1.2772&16.0367&1.8e-4&113(36)&4.8561&16.0367&7.1e-3\cr
    65&93&10.6259&16.0457&4.4e-4&77(37)&42.2455&16.0457&1.9e-2\cr
    129&128&107.0279&16.0480&1.2e-3&33(15)&143.9272&16.0481&3.0e-1\cr
    \bottomrule
    \end{tabular}
    \end{threeparttable}
    \end{center}
}
\end{table}

\Cref{fig:convergence} illustrates the convergence of value of the objective function via iteration numbers for ADMM and RN more intuitively in the case of $\beta=0.5$. We start from the same initial point directly without using what computed from the coarse mesh.
\begin{figure}[H]
\centering
\subfloat{\label{fig:a2}}\includegraphics[width=6cm]{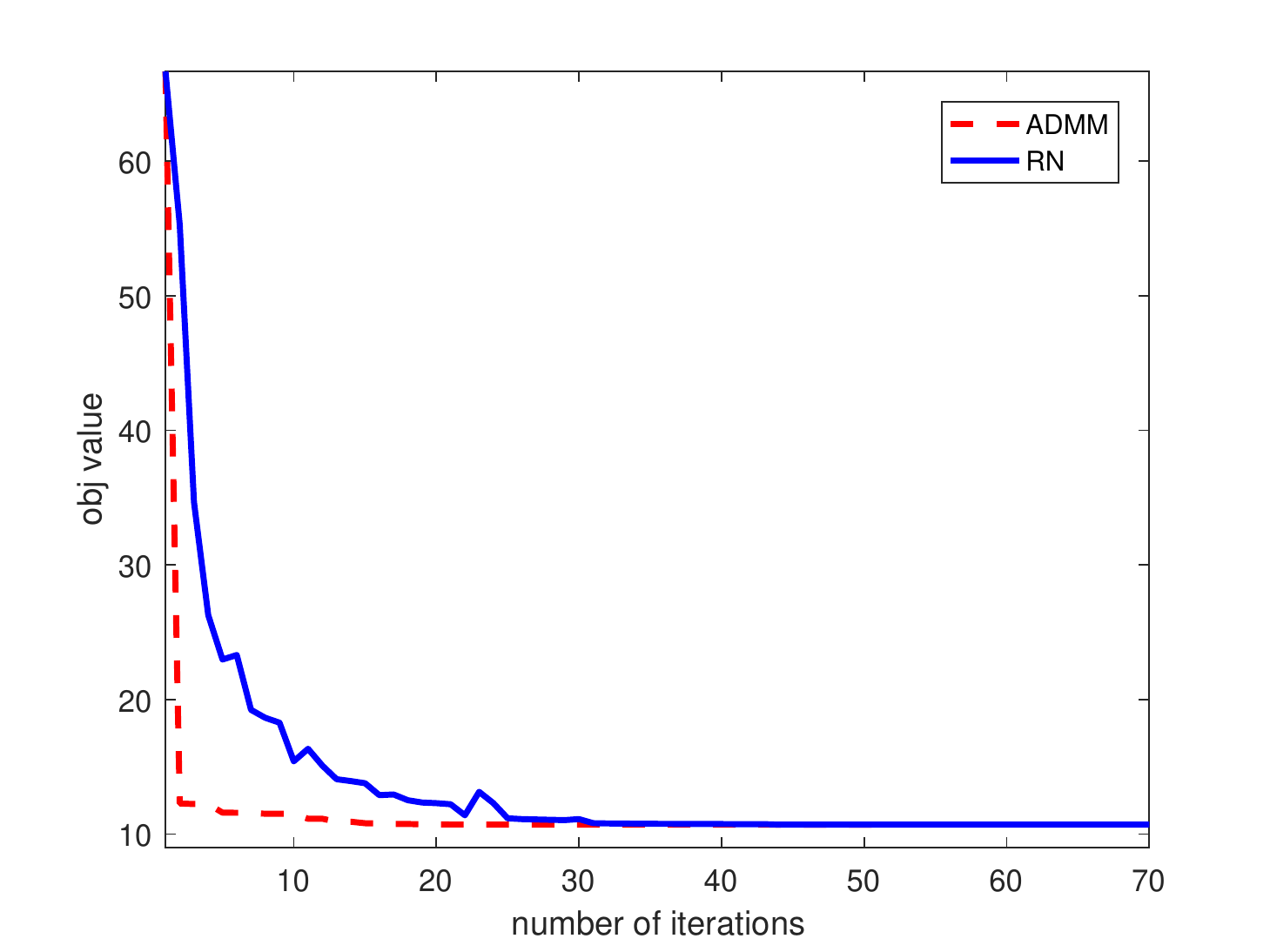}
\subfloat{\label{fig:b2}}\includegraphics[width=6cm]{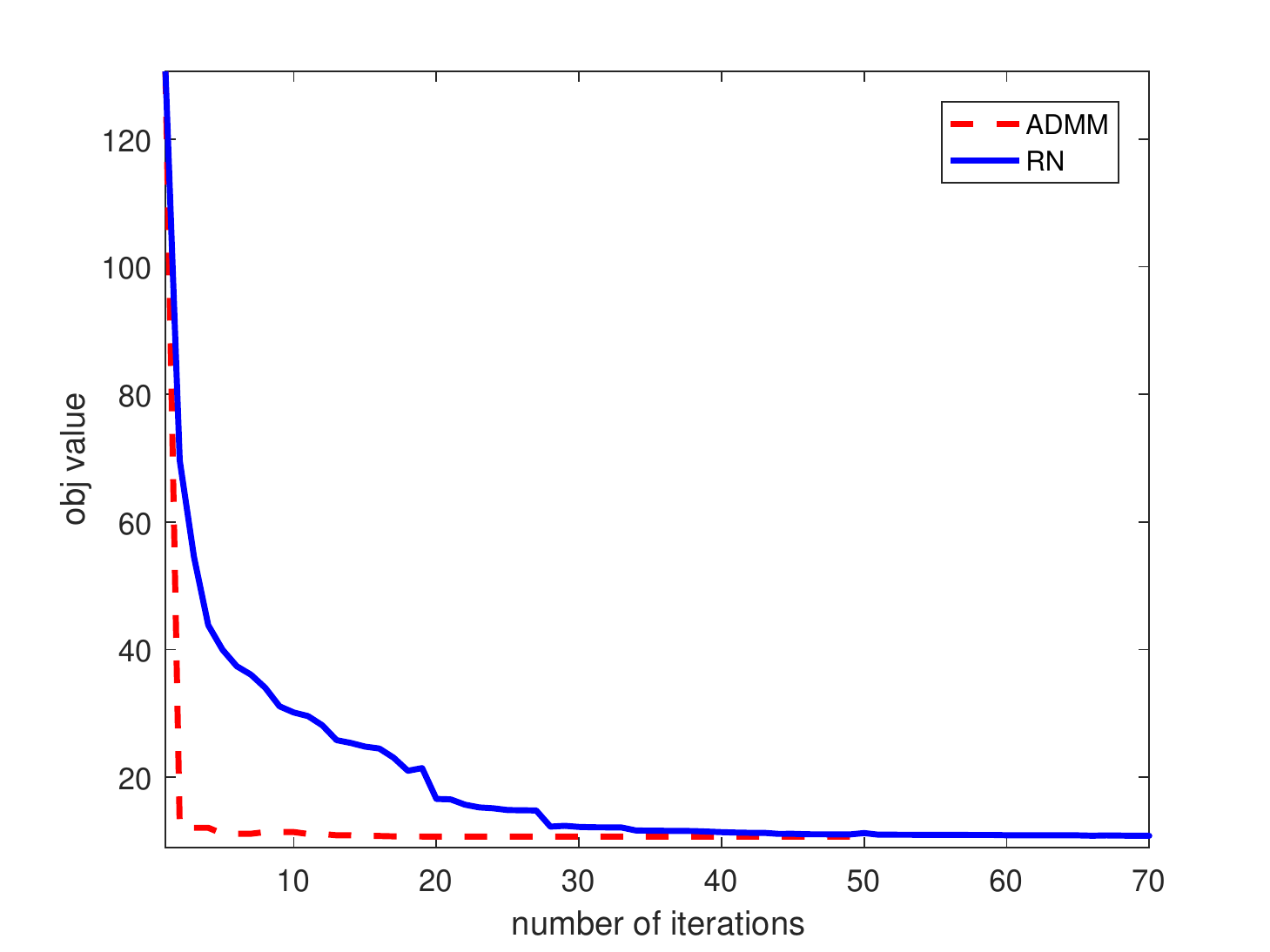}
\caption{Illustration of the objective function value via total iteration numbers for ADMM and RN in the two dimensional case. The left is when the split points along each direction $N=33$, and the right is for $N=65$.}
\label{fig:convergence}
\end{figure}

From the numerical results of the comparison, we found that although ADMM with $\rho$ selected carefully has the possibility to take fewer total iterations, the inner iteration is not as efficient as expected for large scale problems. And this leads to that it takes more time than the RN method. Other than the choice of $\rho$, solving a linear system in each inner iteration for Newton method is the major bottleneck. For the discretized BEC problem, we may take the advantage of the structure of the Laplacian operator to solve the linear system more efficiently. We will not discuss it within this paper.

\section{Concluding Remarks}
\label{sec:conclude}

We have considered a special nonconvex optimization problem over a spherical constraint and characterized it with a nonlinear eigenvalue problem with eigenvector nonlinearity (NEPv). The properties of NEPv were studied. Attention was paid to the smallest eigenvalue, which corresponds to a unique nonnegative (nonpositive) eigenvector. We established the equivalence between this eigenvector and the global optimum, which can help to determine whether a stationary point found by algorithms is a global optimum. Designing algorithms based on this, convergence to the global minimizer of algorithms can be obtained by trivial modification, like the RN method. The ADMM for this nonconvex minimization problem has proven global convergence to the global minimum. We validated our theories by numerical experiments arising in the discretized non-rotating BEC problem.

The results presented in this work depend on the structure of $B$ strongly. However, the extension to more general cases such as the rotating BEC problem seems not to be straightforward. How to solve the problem when this assumption is relaxed is a subject of our future study. Also as already mentioned, another future work will be the improvement of algorithms for solving the BEC-like problems in large scale, including general accelerated schemes of ADMM, such as \cite{he2016convergence}, dealing with the large scale linear system for the subproblem utilizing the structure of discretized Laplacian operator and other parallelizable algorithms.

\begin{remark}
We noticed that Choi et al. \cite{choi2001generalization,choi2002global} have discussed the unique positive solution of NEPv with any fixed $\lambda>\lambda_{min}(B)$ without the spherical constraint under \cref{ass:structure}. Here, we go further to discuss its relationship with the global optimum of a nonconvex optimization.

Choi et al. proved their theories based on the fixed point theory and the Perron-Frobenius theorem for irreducible nonnegative matrices $B^{-1}$. However, because the norm constraint was not considered, they left an open question about the description of the whole spectrum of $\alpha\mathcal{A}x^3+Bx=\lambda x$. They gave an example that $\alpha\mathcal{A}x^3+Bx=\lambda x$ could have eigenvectors with mixed signs. In this paper, we answer it partially. \cref{lem:eigen} further obtains that for all eigenvectors with the same norm and their related eigenvalues, the positive one corresponds to the smallest eigenvalue. On the other hand, for the smallest eigenvalue which has an eigenvector with the same sign, we prove that it is geometrically simple and will not have eigenvectors with mixed signs.
\end{remark}
\section*{Acknowledgments}
The authors would like to thank Professor Xinming Wu, Dr. Jinshan Zeng and Dr. Xudong Li for their inspiration and help.

\bibliographystyle{siamplain}
\bibliography{ref}
\end{document}